\documentclass[a4paper,11pt,reqno]{amsart}

\usepackage{
amsfonts,
amsmath,
amsopn,
amssymb,
amsthm,
bbm,
bbold,
comment,
dsfont,
enumitem,
graphicx,
mathrsfs,
mathtools,
mathabx,
soul,
subfig,
verbatim,
xcolor,
xspace,
latexsym
}

\usepackage{geometry}
\usepackage[hidelinks]{hyperref}
\usepackage[utf8]{inputenc}

\usepackage{rotating}

\usepackage{tikz}
\usetikzlibrary{decorations.pathreplacing}
\usetikzlibrary{shapes.geometric,calc,decorations.markings,math}
\tikzstyle{nodo}=[circle,draw,fill,inner sep=0pt,minimum size=%
1.5mm]
\tikzstyle{infinito}=[circle,inner sep=0pt,minimum size=0mm]

\mathtoolsset{showonlyrefs}

\newtheorem{theorem}{Theorem}
\newtheorem{lemma}{Lemma}[section]
\newtheorem{proposition}[lemma]{Proposition}
\newtheorem{corollary}[lemma]{Corollary}
\newtheorem*{problem}{Problem}

\theoremstyle{remark}
\newtheorem{remark}[lemma]{Remark}
\newtheorem*{remark*}{Remark}

\theoremstyle{definition}
\newtheorem{definition}[lemma]{Definition}

\newcommand{\R}{\mathbb{R}}
\newcommand{\Rd}{\mathbb{R}^2}
\newcommand{\Rp}{\mathbb{R}^+}
\newcommand{\C}{\mathbb{C}}

\newcommand{\Z}{\mathbb{Z}}

\newcommand{\D}{\mathcal{D}}
\newcommand{\V}{\mathcal{V}}
\newcommand{\Eps}{\mathcal{E}}
\newcommand{\G}{\mathcal{G}}

\newcommand{\I}{\mathcal{I}}

\renewcommand{\Re}{\mathrm{Re}}
\renewcommand{\leq}{\leqslant}
\renewcommand{\geq}{\geqslant}

\newcommand{\la}{\lambda}

\newcommand{\x}{\mathbf{x}}

\newcommand{\z}{\mathbf{0}}

\newcommand{\lap}{\Delta}
\newcommand{\na}{\nabla}

\newcommand{\f}[2]{\frac{#1}{#2}}

\newcommand{\scal}[2]{\langle #1,#2\rangle}

\newcommand{\deb}{\rightharpoonup}

\title[NLS ground states on a hybrid plane]{NLS ground states on a hybrid plane}

\author[R. Adami]{Riccardo Adami}
\address{R. Adami: Politecnico di Torino, Dipartimento di Scienze Matematiche ``G.L. Lagrange'', Corso Duca degli Abruzzi, 24, 10129, Torino, Italy.}
\email{riccardo.adami@polito.it}

\author[F. Boni]{Filippo Boni}
\address{F. Boni: Politecnico di Torino, Dipartimento di Scienze Matematiche ``G.L. Lagrange'', Corso Duca degli Abruzzi, 24, 10129, Torino, Italy.}
\email{filippo.boni@polito.it}

\author[R. Carlone]{Raffaele Carlone}
\address{R. Carlone: Universit\`a degli Studi di Napoli Federico II, Dipartimento di Matematica e Applicazioni ``Renato Caccioppoli'', Via Cintia, Monte S. Angelo, 80126, Napoli, Italy.}
\email{raffaele.carlone@unina.it}

\author[L. Tentarelli]{Lorenzo Tentarelli}
\address{L. Tentarelli: Politecnico di Torino, Dipartimento di Scienze Matematiche ``G.L. Lagrange'', Corso Duca degli Abruzzi, 24, 10129, Torino, Italy.}
\email{lorenzo.tentarelli@polito.it}

\date{\today}

\begin{document}


\begin{abstract}
We study existence, nonexistence, and qualitative properties of ground states for a focusing, subcritical Nonlinear Schrödinger Equation on a hybrid plane, consisting of a half-line attached to a plane. Ground states are normalized minimizers of the associated energy, given by Nonlinear Schrödinger energies with contact interactions on the half-line and on the plane, plus a quadratic coupling term. At fixed mass, existence holds if the contact interaction on the half-line is not too repulsive, or the interaction on the plane is sufficiently attractive, or the coupling is strong enough. Nonexistence occurs when both interactions are sufficiently repulsive and the coupling is weak. Moreover, we discuss how the coupling affects the support and the symmetry properties of such ground states. These are the first results for a Nonlinear Schrödinger Equation on a mixed-dimensional manifold.
\end{abstract}

\maketitle

\vspace{-.5cm}

\noindent {\footnotesize \textul{Keywords:} hybrids, standing waves, nonlinear Schr\"odinger, ground states, delta interaction, radially symmetric solutions, rearrangements.}


\section{Introduction}

The quantum dynamics on domains of mixed dimensionality was first investigated by Exner and  $\check{\text{S}}$eba in the seminal work
\cite{ES-87}, where the authors considered a manifold  made of a plane glued to the end of a half-line as in Figure \ref{fig-hyb}. The origin of both the half-line and the plane was set at the junction. The resulting structure, here denoted by $\mathcal I$, was called \emph{hybrid plane}. According to the physical theory, a state of a particle confined on $\I$ consists of a couple of functions   
\begin{equation}
U = (u,v), \qquad u \in L^2 (\R^+), \, v \in L^2 (\R^2).
\label{states}
\end{equation}

In the same work, the authors described all matching conditions that link $u$ and  $v$ at the junction
and realize the Laplacian as a self-adjoint operator in $L^2 (\I) := L^2 (\R^+) \oplus L^2 (\R^2)$.
This problem and the  techniques used in \cite{ES-87} belong to the field of the analysis of linear operators.

In the present paper we initiate the analysis of nonlinear problems on hybrid structures. Here we focus on the search for conditions of existence of ground states for the standard focusing Nonlinear Schr\"odinger (NLS) Energy on $\mathcal I$ in the $L^2$-subcritical case. 

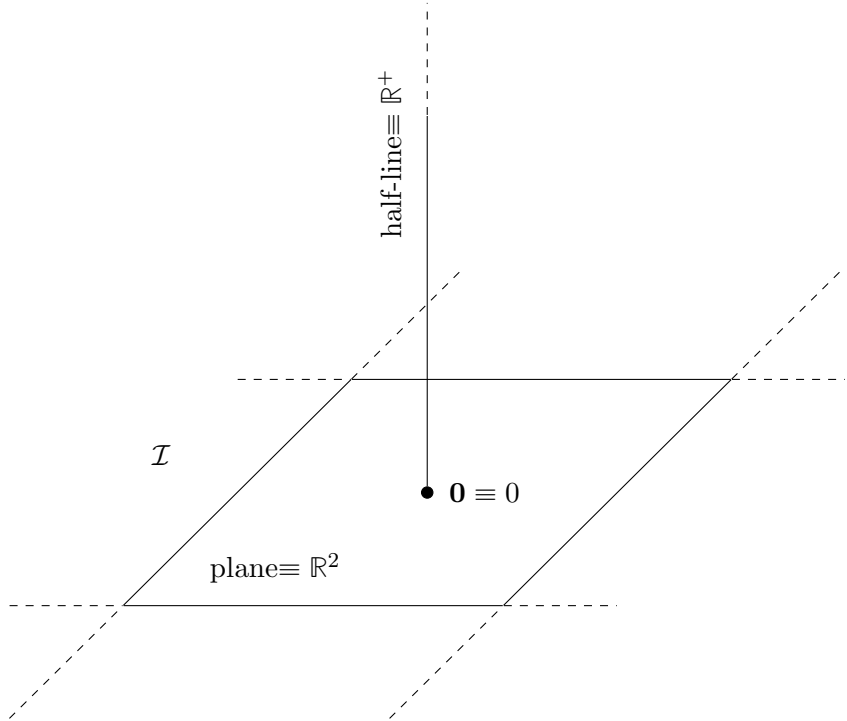
\begin{figure}
\centering
\begin{tikzpicture}[xscale= 0.5,yscale=0.5]
\node at (-3,-3) [infinito] (-3-3) {};
\node at (7,-3) [infinito] (7-3) {};
\node at (-3,0) [infinito] (-30) {};
\node at (0,0) [infinito] (00) {};
\node at (10,0) [infinito] (100) {};
\node at (13,0) [infinito] (130) {};
\node at (8,3) [nodo] (83) {};
\node at (3,6) [infinito] (36) {};
\node at (6,6) [infinito] (66) {};
\node at (16,6) [infinito] (166) {};
\node at (19,6) [infinito] (196) {};
\node at (9,9) [infinito] (99) {};
\node at (19,9) [infinito] (199) {};
\node at (8,13) [infinito] (813) {};
\node at (8,16) [infinito] (816) {};
\node at (1,4) [infinito] (14) {$\I$};
\node at (4,1) [infinito] (41) {plane$\equiv\Rd$};
\node at (9.5,3) [infinito] (8525) {$\z\equiv0$};
\node at (7,12) [infinito] (712) {\begin{turn}{90} 
half-line$\equiv\Rp$
\end{turn}};
\draw [-] (00) -- (100);
\draw [-] (00) -- (66);
\draw [-] (100) -- (166);
\draw [-] (66) -- (166);
\draw [-] (83) -- (813);
\draw [dashed] (-3-3) -- (00);
\draw [dashed] (7-3) -- (100);
\draw [dashed] (-30) -- (00);
\draw [dashed] (100) -- (130);
\draw [dashed] (36) -- (66);
\draw [dashed] (166) -- (196);
\draw [dashed] (66) -- (99);
\draw [dashed] (166) -- (199);
\draw [dashed] (813) -- (816);
\end{tikzpicture}
\caption{ The hybrid plane $\I$. The origin of the coordinates on both the half-line and the plane is set at the junction.}
\label{fig-hyb}
\end{figure}
We construct such energy 
heuristically, however we will 
show in Appendix \ref{sec:operatorES} that it coincides with the sum of the energy associated to the Laplacian 
defined in \cite{ES-87} with a nonlinear power term.

To construct the NLS energy on $\I$, we start from the well-known expression of the focusing standard NLS energy (\cite{GV-79,GV-79bis,C-03})
\begin{equation} \label{standard}
E_{NLS} (w,X)  =  \frac 1 2 \| \nabla w \|^2_{L^2(X)}
- \f 1 s \| w \|^s_{L^s(X)}, \qquad w \in H^1 (X), 
\end{equation}
where $X$ can be
the half-line or the plane; then, we assume that the effect of the contact between the plane and the half-line is concentrated at the junction, giving rise to one contact interaction on the half-line, to another on the plane, and to a coupling between the two components. As a result, the contribution of the half-line to the energy is
\begin{equation} \label{e1}
E_\alpha(u,\R^+) = \frac 1 2 \| u' \|_{L^2 (\R^+)}^2 - \frac 1 p \| u \|^p_{L^p (\R^+)} +
\frac \alpha 2 | u (0) |^2, \qquad u \in H^1 (\R^+),
\end{equation}
where the real number $\alpha$ denotes the strength of the contact interaction on the half-line, which coincides with a Dirac's delta
potential at the origin. For this reason, in what follows we will reserve the name delta interaction to the contact interaction on the half-line.

Concerning the contribution of the planar component $v$ to the energy on $\I$, 
it is well-known that a Dirac's delta term like in \eqref{e1}, say $\rho | v (\bf 0) |^2$, is not controlled by the gradient term in \eqref{standard} and then it does not give rise to a closed and lower bounded form, making the energy ill-defined (for details see \cite{AGHKH-88}). Thus, in order to define a contact interaction in two dimensions, one can make use of the theory of self-adjoint extensions of hermitian operators by von Neumann and Krein \cite{RS-79}, and, as a result, one finds that
the presence of a contact interaction imposes a particular structure to the function $v$. Precisely, $v$ decomposes into the sum of a  regular and a singular part, namely
\begin{equation} \label{decv}
v  =  \phi + q {\frac{K_0}{2 \pi}}, \, \qquad \phi \in H^1 (\R^2), \, q \in \C, 
\end{equation}
where the complex number $q$ is called the charge of the function $v$ and $K_0$ is the Macdonald function of order zero, whose asymptotic behaviour at the origin is
$$ K_0 (\x) =  - \log |\x| + o (\log |\x| ), \qquad \x \to 0.$$
Notice that $K_0$ is not in $H^1 (\R^2)$, so the decomposition \eqref{decv} is consistent. 

The energy of the function $v$ in \eqref{decv} reads (\cite{ABCT-22,AGHKH-88})
\begin{equation} 
\label{e2}
E_\rho(v,\R^2) =  \frac 1 2 \| \phi \|_{H^1 (\R^2)}^2 
- \frac 12 \|v \|_{L^2 (\R^2)}^2 
+
\left( \rho + \frac{\gamma - \log 2}{2 \pi} \right)
\frac{|q|^2}{2} - \frac 1 r \| v \|^r_{L^r (\R^2)},
\end{equation}
where $\gamma$ is the Euler-Mascheroni constant and the real number $\rho$ represents the strength of the contact interaction. More precisely, $-\frac 1 {2 \pi \rho}$ is the scattering length of the contact interaction, so that the free Laplacian is recovered by setting $\rho = \infty$, that formally forces $q = 0$ in \eqref{e2} and therefore $v = \phi$.

Given the contribution  of the half-line in \eqref{e1} and that of the plane in \eqref{e2}, to complete the definition of the energy one needs to introduce the coupling between the two components.
We choose such a coupling
as quadratic in the main orders of $u$ and $v$ near the junction, namely $u (0)$ and $q$. The coupling term equals then $- \beta \Re (q \overline{u(0)})$, with $\beta \geq 0$. The restriction on $\beta$ does not affect the generality, since for a complex $\beta$ it is sufficient to multiply the function $u$ by the phase of $\beta$ in order to recover an expression of the energy with a nonnegative $\beta$. 

The resulting expression for the energy on the hybrid is then
\begin{equation}
    \label{ei}
    E (U)  : =  E_\alpha (u, \Rp) + E_\rho (v, \R^2) - \beta \Re (q 
    \overline{u(0)}), \qquad U = (u,v). 
\end{equation}
Let us remark that the energy of the half-line $E_\alpha (\cdot, \R^+)$ is the restriction of $E$ to states with $v = 0$ and the energy on the plane $E_\rho$ is the restriction of $E$ to states with $u = 0$.

We are now ready to introduce the notion of ground state.
By ground state for the energy \eqref{ei} at mass $\mu$ we mean a function $U$ in the energy domain
\begin{equation} \label{energydomain}
\mathcal D: = \left\{ U = (u, v): \, u \in H^1 (\R^+), \, v = \phi + q K_0/ 2\pi, \, \phi \in H^1 (\R^2), \, q \in \C\right\},
 \end{equation}
that minimizes $E$ among the functions satisfying the mass constraint 
\begin{equation}
\label{mass}
\| u \|_{L^2 (\R^+)}^2 + \| v \|_{L^2 (\R^2)}^2  =  \mu.
\end{equation}
We use the notation $\mathcal D_\mu$ for the intersection of the energy domain $\mathcal D$ with the mass constraint.

So, in order for a ground state to exist, the infimum 
\begin{equation} \label{infimum}
\mathcal E (\mu)  : =  \inf_{\mathcal D_\mu} E (U) 
\end{equation}
must be finite, i.e. $\mathcal E (\mu) > - \infty$, and attained, i.e. there must be $U \in \mathcal D_\mu$ such that
$$ E (U) = \mathcal E (\mu). $$

Finally, as already anticipated, here we restrict to $L^2$-subcritical nonlinearities, i.e. we consider powers bounded as $2 < p < 6$ and $2 < r < 4$. The meaning of this limitation is illustrated by the Gagliardo-Nirenberg inequality
\begin{equation}
\label{eq-gigen}
\| w \|_{L^s (\R^n)}^s  \leq  C \| \nabla w \|_{L^2 (\R^n)}^{n \left( \frac s 2 - 1 \right)} \| w \|_{L^2 (\R^n)}^{n + s \left( 1 - \frac n 2 \right)},
\end{equation}
that, provided $s < 2 + 4/n$, yields 
$$
E_{NLS} (w,\R^n)  \geq  \frac 1 2 \| \nabla w \|^2_{L^2(\R^n)} - C \| w \|_{L^2 (\R^n)}^{n + s \left( 1 - \frac n 2 \right)} \| \nabla w \|_{L^2(\R^n)}^{n \left( \frac s 2 - 1 \right)},
$$
so that the constrained standard energy is coercive. In the proof of Theorem 1 we show that  \eqref{eq-gigen} yields coercivity for the energy \eqref{ei} also, under the constraint \eqref{mass} and the subcriticality conditions $2 < p < 6$ and $2 < r < 4$.

Of course coercivity guarantees that the infimum of the functional is finite, but in general does not ensure that it is attained. In fact, while in $\R^n$ subcriticality entails existence of ground states for every mass, this is not true in more complicated domains, like for instance  metric star graphs with at least three half-lines (\cite{ACFN-12,AST-15}).  

In the construction of the energy \eqref{ei} we introduced two different nonlinearity powers, $p$ in the half-line and $r$ in the plane, to enclose a larger spectrum of problems and to cover all subcritical cases.

Summing up, the problem we analyze is the following:
\begin{problem}
Given $\alpha, \rho \in \R$, $\beta \geq 0$, $p\in(2,6)$, $r\in(2,4)$ and  $\mu > 0$,
establish the existence or the nonexistence and, in case of existence, some qualitative features for the ground states at mass $\mu$ of the energy \eqref{ei}, namely the
minimizers of the functional
\begin{equation}
    \label{e} \begin{split}
E (U)  =  & \,E_\alpha (u, \R^+) + E_\rho (v, \R^2) -  \beta \Re (q \overline{u(0)}) \\
 =  &
 \,\frac 1 2 \| u' \|_{L^2 (\R^+)}^2 + \frac \alpha 2 | u (0) |^2 - \frac 1 p \| u \|^p_{L^p (\R^+)}
+ \frac 1 2 \| \phi \|_{H^1 (\R^2)}^2 
- \frac 12 \|v \|_{L^2 (\R^2)}^2  \\ &
+
\left( \rho + \frac{\gamma - \log 2}{2 \pi} \right)
\frac{|q|^2}{2} - \frac 1 r \| v \|^r_{L^r (\R^2)}
-  \beta \Re (q \overline{u(0)}),
\end{split}
\end{equation}
where:
\begin{itemize}
\item $u \in H^1 (\R^+)$;
\item $v = \phi + q {K_0} /{2 \pi},$ with $\phi \in H^1 (\R^2),\, q \in \C$;
\item $U = (u,v)$ satisfies the mass constraint \eqref{mass}.
\end{itemize}
\end{problem}


Let us quickly summarize our results, to be detailed in Section \ref{sec:results}. According to the formulation of the Problem, some of them focus on the issue of existence or nonexistence of ground states, while the others concern the shape of such ground states, provided they exist.  

The key step for proving existence of ground states lies in showing that the only possible lack of compactness for a minimizing sequences is given by the escape at infinity through the half-line. In other words, the escape at infinity along the plane is not an issue as it is never energetically convenient. This is due to the fact that the plane alone with a contact interaction always traps a ground state \cite{ABCT-22}.
Therefore, in order for a ground state to exist, the energy threshold to be overcome is the energy of the one-dimensional soliton at mass $\mu$ (Theorem \ref{thm:1}).
Conversely, if the energy of every state is strictly larger than the energy of the one-dimensional soliton with the same mass, then ground states at that mass cannot exist.

Concrete results of existence derive from Theorem \ref{thm:1}. First, in Corollaries \ref{smallmu} and \ref{largemu} we obtain that, for every choice of the parameters $\alpha \in \R,\, \rho \in \R,\, \beta \geq 0,\,p\in(2,6),\,r\in(2,4)$, a ground state exists for small and for large mass. This result is obtained through  asymptotic analysis and does not provide significant estimates of the mass thresholds. These are given in Corollary \ref{interplay}, where we provide explicit thresholds $\mu_{p,r}$ for the mass and $r^\star$ for the nonlinearity, so that existence is guaranteed whenever $\mu$ and $r$ are either both above or both below the respective thresholds.
On the other hand, for fixed mass there exists a ground state as one of the following hypotheses is satisfied:
\begin{enumerate}
\item 
if $\alpha$ is small enough, depending only on the power $p$ (Corollary \ref{smallalfa}). Since the threshold in $\alpha$ is positive for every $p$, this means that a not too repulsive delta potential at the origin of the half-line guarantees, alone, the presence of a ground state;
\item if $\rho$ is not too large, with a threshold that depends on all other parameters (Corollary \ref{limits}). Recalling that in the linear Schr\"odinger Equation the two-dimensional contact interaction 
always has a ground state whose energy increases with $\rho$
\cite{AGHKH-88}, this condition means that,
even when the delta interaction at the half-line is strongly repulsive, a sufficiently binding contact interaction on the plane captures a ground state for the hybrid;
\item if $\beta$ is large enough, with a threshold that depends on all other parameters (Corollary \ref{limits}). This means that, even if all other parameters play an unfavorable role, a strong coupling between the halfline and the plane can trap a ground state. In fact, the coupling always has a binding effect.
\end{enumerate}

Conversely, if we are not in the hypotheses of the previous corollaries that guarantee existence, namely if $\alpha$ and $\rho$ are large and $\beta$ is small, then there are no ground states (Theorem \ref{thm:2}).

 We point out that our classification does not cover the case $r = \frac{6p-4}{p+2}$, since this condition results in the same scaling law for the energies of the ground states on the free line and the free plane as functions of the mass. So, in order to figure out which component of $\I$ is energetically more convenient, one should use the numerical value of the energy of the soliton in the half-line and in the plane, but the latter is currently unknown. Furthermore, as we specify more precisely in Section \ref{sec:results}, our results do not cover the whole spectrum of the parameters.

The second kind of results concerns some basic features of the ground states. We distinguish the case of decoupled half-line and plane, i.e. $\beta = 0$, and the case with a coupling $\beta > 0$.
In the first case, every ground state is supported either on the half-line, where it is a  soliton tail, or on the plane, 
where it coincides with 
a ground state for the NLS with a point interaction of strength $\rho$ (Theorem \ref{thm:3}). If the plane and the half-line are coupled, then every ground state possesses non-trivial components on the plane and on the half-line, each being a ground state at some mass for the NLS with a suitable point interaction at the origin (Theorem \ref{thm:4}). In particular, every ground state displays a logarithmic singularity at the origin of the plane and its two components are connected at the junction by the
boundary conditions \eqref{bc} singled out in \cite{ES-87}. 

As for the techniques, our proofs rely on the interplay between
direct and indirect methods of the Calculus of Variations, together with techniques developed for studying contact interactions in Quantum Mechanics.
Our starting point is the information obtained separately for the
problem on the half-line \cite{BC-23} and the plane \cite{ABCT-22,FGI-22}. However, most of the proofs here require more advanced techniques. This is mainly because when studying the existence of ground states on the hybrid, the problems on the two components are separately unconstrained, being the constraint only on the hybrid as a whole. In particular, the proof of Theorem \ref{thm:1} is more technical than the analogous proof in \cite{ABCT-22}, even though we managed to simplify the use of the  Gagliardo-Nirenberg inequalities, adequately generalized. Moreover, to prove the features of ground states in the coupled case (Theorem \ref{thm:4}), we follow the strategy introduced in \cite{ABCT-22}, i.e. we move to a problem that rephrases the minimization of the action on the Nehari manifold. In such a problem, applying a suitable version of the rearrangement inequalities proves meaningful, while in the original problem this is much harder. However, as Remarks \ref{rem-problemi} and \ref{rem-problemi2} explain, in this context such strategy presents a major technical obstacle, which requires further non-trivial efforts to be solved. This is a consequence of choosing different nonlinearity powers on the half-line and the plane.

\smallskip

Our model opens a new research line in a topic that in recent years has prompted a widespread and still growing interest,  the Nonlinear Schr\"odinger (NLS) equation on non-standard domains.
In such topic a prominent role has been played by the study of the NLS on {metric graphs} (see e.g.  \cite{BK-13} for a rigorous definition). The reason is twofold. On the one hand, such a model proved to be an effective description of the dynamics of {Bose-Einstein Condensates} (BEC) in branched traps  \cite{GW-16,LLMOSCMWCCT-14}; which is, among other things, the basic framework for the study of emerging quantum technologies such as {atomtronics} \cite{ABBMMKAAAAAB-21}. On the other hand, from a mathematical perspective, it gives rise to several new and unexpected phenomenologies for the same problem on Euclidean domains.

Consequently, the mathematical literature on the topic has become large in a relatively short time, especially concerning the study of the standing waves for the {focusing} NLS. For the sake of brevity, we only mention a selection of the most recent results on compact graphs \cite{CJS-23,DGMP-20}, graphs with half-lines \cite{ABD-20,BCJS-23,BD-21,BD-22,DDGS-23,DDGS-23bis,DST-20,NP-20,PS-22,PSV-21} and graphs with infinitely many edges \cite{ADST-19,D-23,DST-20bis,DT-22}, even with relativistic corrections \cite{BCT-19,BCT-21,BCT-21bis}.

In this perspective, the study of the NLS on the hybrid plane $\I$ and, more generally, on connected manifolds consisting of pieces of different dimension, also known as {quantum hybrids}, appears as a natural generalization of the analogous investigations carried out for quantum graphs. However, the complexity of the techniques grows considerably.

In this regard, it is worth mentioning that the discussion on the NLS ground states on hybrids shares more features with the study of the ground states for the NLS with {point defects} of delta type, nowadays almost completely understood for the problem on the real line  \cite{AN-13,ANV-13,BD-21,FJ-08,FOO-08,LFFKS-08} and only recently addressed in the cases of the half-line \cite{BC-23} and of the plane and the space \cite{ABCT-22,ABCT-22bis,CFN-21,CFN-22,FGI-22,FN-22}.

Due to the structure of the hybrid plane $\I$, we take advantage of the results obtained in \cite{ABCT-22,BC-23}, which represent  preliminary investigations aimed at the present work. 

\medskip

Finally, the paper is organized as follows:
\begin{itemize}[label=$\ast$]
 \item in Section \ref{sec:results} we state and comment the results and give some corollaries;
 \item in Section 3 we provide the proof of the main result (Theorem \ref{thm:1}), that gives necessary and sufficient conditions for the existence of ground states at mass $\mu$.
 \item in Section \ref{sec-exnon} we present the proof of a nonexistence result (Theorem \ref{thm:2}). 
 \item Section \ref{sec-properties} addresses the proof of Theorem \ref{thm:4}, concerning the properties of ground states when $\beta>0$.
\end{itemize}
The paper ends by two appendices. The first shows that  the self-adjoint operator $H$ associated with the quadratic part of \eqref{e} is, apart from a factor $2$, that constructed in \cite{ES-87}, and the second details the spectrum of $H$ in view of the importance of the linear ground state for the non linear problem.


\section{Main results} 
\label{sec:results}

Before stating the results, we point out that there exist three negative upper bounds for $\mathcal E(\mu)$.

\begin{remark} [A priori upper bounds]
\label{rem:apriori}
 The function $\mathcal E (\mu)$ is always negative. More precisely, for every $\mu$ it lies below the following three levels.

\begin{enumerate}
\item
Denote by $\varphi_\mu$ the unique positive and even ground state at mass $\mu$ of the functional \eqref{standard} with $X=\R$ and $s=p$, whose existence is a classical result (\cite{ZS-71,L-82,L-84,L-84bis,C-03}). Furthermore,
\begin{equation} \label{standard1}
E_{NLS} (\varphi_\mu, \R)  =  - \theta_p \mu^{\frac{p+2}{6-p}}, \qquad \theta_p:=-E_{NLS} (\varphi_1, \R) > 0.
\end{equation}
Now consider truncations of $\varphi_\mu$ on the half-line
$\varphi_{\mu}^{(n)} = N^{(n)} \chi_{R^+} \varphi_\mu (x-n)$, where the factor $N^{(n)}$ normalizes the mass to $\mu$. Clearly,  
$\varphi_{\mu}^{(n)} (0 ) = \varphi_\mu (n)$ vanishes and $N^{(n)}\to1$, as $n \to +\infty$, so
$$ E_{\alpha} (\varphi_{\mu}^{(n)}, \R^+)  =  E_{NLS} (\varphi_\mu, \R) + o (1)  \to  - \theta_p \mu^{\frac {p+2}{6-p}}, \qquad n \to +\infty.$$
Defining $\Phi_{\mu}^{(n)} = (\varphi_\mu^{(n)}, 0)$ in $\D_\mu$, one gets
\begin{equation} \label{bound1}
 \mathcal E (\mu)  \leq  \lim_n E (\Phi_\mu^{(n)})  =  
\lim_n E_{\alpha} (\varphi_{\mu}^{(n)}, \R^+)  =  - \theta_p \mu^{\frac {p+2}{6-p}}.
\end{equation}
Note that this bound may not be attained.

\item 
Analogously, denote by $\xi_\mu$ the unique positive and spherically symmetric ground state at mass $\mu$ of the functional \eqref{standard} with $X=\R^2$ and $s=r$, whose existence is a classical result too (\cite{L-82,L-84,L-84bis,C-03}). It satisfies
\begin{equation} \label{standard2}
E_{NLS} (\xi_\mu,\R^2)  =  - \tau_r \mu^{\frac 2 {4-r}},
\qquad \tau_r:=-E_{NLS} (\xi_1, \R^2) > 0.
\end{equation}
In addition, denote by $v_\mu$ a positive and spherically symmetric ground state at mass $\mu$ of the functional $E_\rho (\cdot, \R^2)$. In \cite{ABCT-22} we proved its existence and also that
\begin{equation} \label{comparison2}
E_\rho ( v_\mu, \R^2)  <  E_\rho (\xi_\mu, \R^2)
 =  E_{NLS} (\xi_\mu, \R^2),
\end{equation}
since in a decomposition like \eqref{decv} the function $\xi_\mu$ has no singular part. Then, introducing the state $\Upsilon_\mu = (0, v_\mu)$ we get
\begin{equation} \label{bound2}
\mathcal E (\mu)  \leq  E (\Upsilon_\mu) 
 =  E_\rho (v_\mu, \R^2)  <  E_\rho (\xi_\mu, \R^2)
 =  - \tau_r \mu^{\frac 2 {4 -r}}.
\end{equation}

\item
Finally, given $\mu > 0$ we denote by $\Psi_\mu$ the unique positive ground state at mass $\mu$ of twice
the quadratic part of the energy \eqref{e}, namely 
\begin{align} \label{quadratic}
Q (U)  = & \,  \| u' \|_{L^2(\R^+)}^2 + \alpha | u(0) |^2 +\nonumber\\[.2cm] & \, \| \phi \|_{H^1 (\R^2)}^2 
- \|v \|_{L^2 (\R^2)}^2 
+
\left( \rho + \frac{\gamma - \log 2}{2 \pi} \right)
|q|^2-  2\beta \Re (q \overline{u(0)}).
\end{align}
Clearly, $\Psi_\mu = \sqrt \mu \Psi_1$ and $Q (\Psi_\mu) = - E_{\rm{lin}}\mu < 0$, where we denoted by 
\begin{equation}
\label{Elin}
E_{\rm{lin}}:=-\inf_{\substack{U\in \D\\U\neq (0,0)}}\frac{Q(U)}{\|U\|_{L^2(\I)}^2}=-Q(\Psi_1).
\end{equation}
See Appendix \ref{sec:Elin} for a proof that $-E_{\rm{lin}}$ is actually the least eigenvalue of the self-adjoint operator associated with $Q$ and is negative.  As a consequence, one immediately sees that
\begin{equation}
\mathcal E (\mu)  \leq  E (\Psi_\mu)  <  
\frac12Q (\Psi_\mu)  = 
-\frac{E_{\rm{lin}}}{2} \mu.
\label{bound3}
\end{equation}
\end{enumerate}
We stress that, while bounds \eqref{bound1} and \eqref{bound2} do not depend on the parameters $\alpha, \rho$ and $\beta$, the factor $E_{\rm{lin}}$ does.
\end{remark}

An immediate consequence of \eqref{bound3} is that a minimizing sequence at mass $\mu$ for the functional $E$ cannot completely disperse and lose the whole mass in the weak limit, since the $L^p(\R^+)\oplus L^r(\R^2)$ norm would vanish and the limit of the energy could not be lower than $- \frac{E_{\rm{lin}}}{2} \mu$, contradicting \eqref{bound3}. Moreover, we prove in Theorem \ref{thm:1} that the possibility of splitting a fraction of the  mass is not energetically convenient. Therefore, by concentration-compactness method we conclude that either a minimizing sequence strongly converges to a ground state, or it escapes away to infinity.

In turn, the escape may take place in two ways: on the plane or along the half-line. In either case, the contribution of the contact interaction to the energy becomes negligible, so that, in order to minimize the energy, the sequence has to approximate asymptotically a soliton. However, if a minimizing sequence escaped through the plane, then its energy would converge to that of $\xi_{\mu}$, and
so  $\Xi_{\mu} = (0, \xi_\mu)$
would be a ground state at mass $\mu$ for \eqref{e}.

Hence,
the only situation in which ground states do not exist occurs when
minimizing sequences escape along the half-line and asymptotically mimic the one-dimensional soliton $\varphi_\mu$. Conversely, in order to ensure that a ground state exists, it is sufficient to exhibit a competitor whose energy is lower than the energy of the one-dimensional soliton with the same mass. This is summarized in the theorem below.

\begin{theorem}[A necessary and sufficient condition for existence]
\label{thm:1}
A ground state at mass $\mu$ for the energy \eqref{e} on the hybrid plane $\I$ exists if and only if there is a function $W \in \mathcal D_\mu$ such that 
\begin{equation} \label{excrit}
E (W)  \leq  E_{NLS} (\varphi_\mu, \R)  =  - \theta_p \mu^{\frac{p+2}{6-p}}
\end{equation}
where $\varphi_\mu$ is a ground state at mass $\mu$
for the standard NLS energy $E_{NLS}(\cdot, \R)$.
\end{theorem}

As a first consequence of Theorem \ref{thm:1}, we remark that a ground state always exists provided that the mass is small enough.
This conclusion comes straightforwardly from Theorem \ref{thm:1} and the bound \eqref{bound3}. Indeed,
since \eqref{bound3} is linear in $\mu$, while the bound in \eqref{excrit} is superlinear,
for small mass one has
$$ \mathcal E (\mu)   <  - \frac{E_{\rm{lin}}}{2} \mu
 <  - \theta_p \mu^{\frac{p+2}{6-p}}.
$$
So we proved the corollary below.

\begin{corollary}[Existence for small $\mu$] \label{smallmu}
Fix $\alpha,\, \rho \in \R, \, \beta \geq 0,\,p\in(2,6)$. Then, there exists $\mu^\star (p,\alpha, \rho, \beta) > 0$ such that for every $\mu \in(0, \mu^\star)$ there is a ground state at mass $\mu$ for the energy \eqref{e}.
\end{corollary}

In view of point (2) of Remark \ref{rem:apriori}, the second consequence of Theorem \ref{thm:1} is that, 
if the infimum of the free NLS energy \eqref{standard} on the plane is smaller than the analogous quantity on the line, then \eqref{excrit} is satisfied and a ground state exists. Such a condition amounts to
\begin{equation} \label{freecompar}
\theta_p \mu^{\frac{p+2}{6-p}}  <  \tau_r \mu^{\f 2 {4-r}},
\end{equation}
and gives for the mass the threshold value
\begin{equation} 
\label{massthreshold}
\mu_{p,r}  =  \left( \frac{\tau_r}{\theta_p} \right)^{\frac{pr-4p-6r+24}{6p -2r -pr -4}}.
\end{equation}
Since the exponent at the l.h.s. of \eqref{freecompar} is smaller than the exponent at the r.h.s. if and only if
\begin{equation}
\label{r_star}
r >  r^\star : = \f {6p-4}{p+2}
\end{equation}
the next corollary follows.

\begin{corollary} \label{interplay}
Fix $p\in(2,6), \, r\in(2,4)$.
If
\begin{equation} 
\label{eq:r-rstar-mu-range}
\begin{array}{c}
\displaystyle r< r^\star\quad \text{and}\quad \mu\in(0,\mu_{p,r}]\\[.1cm]
\text{or}\\[.1cm]
\displaystyle r> r^\star\quad\text{and}\quad \mu\geq \mu_{p,r},
\end{array}
\end{equation}
then there is a ground state at mass $\mu$ for the energy \eqref{e}.
\end{corollary}
Unfortunately, we are not able to treat the case $r = r^\star$ since, to our knowledge, a comparison between $\theta_p$ and $\tau_r$ is not available.

The third consequence of Theorem \ref{thm:1} is that, whenever a ground state $\varphi_{\alpha, \mu}$ at mass $\mu$ for the functional \eqref{e1} exists, condition \eqref{excrit} is fulfilled by the state
$\Phi_{\alpha,\mu} : = (\varphi_{\alpha, \mu}, 0)$. Indeed one has $\| \Phi_{\alpha, \mu} \|_{L^2 (\I)}^2 = \mu$ and 
$$ E (\Phi_{\alpha, \mu}) = E_\alpha (\varphi_{\alpha, \mu}, \R^+)  \leq  E_{NLS} (\varphi_\mu, \R).$$
Thus, one inherits the next result, proved in \cite[Theorems 1.2 and 1.3]{BC-23} and summarized in the corollary below.

\noindent
\begin{corollary}[Existence for large $\mu$] \label{largemu}

Fix $\alpha \in \R,\,p\in(2,6)$. Therefore:
\begin{enumerate}
\item when $\alpha \leq 0$ there is a ground state at mass $\mu$ for the energy \eqref{e} for every $\mu > 0$;
\item when $\alpha > 0$ and $p\in(2,4]$, if $\mu > \mu_p (\alpha)$, where $\mu_p (\alpha) > 0$ is the mass of the one-dimensional soliton that solves 
$$ u'' + |u|^{p-2}u  =  \alpha^2 u, $$
then there is a ground state at mass $\mu$ for the energy \eqref{e};
\item when $\alpha > 0$ and $p\in(4,6)$ there exists $\widetilde \mu \in(0,\mu_p (\alpha))$ such that, if $\mu > \widetilde \mu$, then there is a ground state at mass $\mu$ for the energy \eqref{e}.
\end{enumerate}
\end{corollary}

At this point we know that there exist ground states for small and large values of the mass. What happens in the intermediate region of masses may depend on the values of the other parameters of the problem. 

The next corollary shows that the existence of ground states for fixed $\mu$ and $p$ is ensured provided $\alpha$ is negative or below a certain treshold (depending on both $\mu$ and $p$). This result is obtained just rephrasing \cite[Proposition 1.4]{BC-23}, that discusses the existence of ground states at mass $\mu$ for \eqref{e1} as $\alpha$ varies. 

\begin{corollary}[Existence for small or negative $\alpha$]\label{smallalfa}
For any $\mu > 0$ and any $p\in(2,6)$ there is a number $\alpha_p (\mu) > 0$ such that,
if 
\begin{equation}
\label{eq:alpha-mu}
\begin{array}{c}
\displaystyle \alpha < \alpha_p (\mu) \\[.1cm]
 \text{or}\\[.1cm]
\displaystyle \alpha = \alpha_p(\mu)\quad \text{and}\quad 4 < p < 6,
\end{array}
\end{equation}
then there exists a ground state at mass $\mu$ for the energy \eqref{e}. Furthermore, $\alpha_p$ is a strictly monotonically increasing function of the mass $\mu$ and satisfies
\begin{equation}
        \label{alfa}
        \alpha_p (\mu)  \ \left\{ \begin{array}{cc} 
              : = &  C_p \mu^{\frac{p-2}{6-p}}, \qquad 2 < p \leq 4 \\ & \\ > & C_p \mu^{\frac{p-2}{6-p}}, \qquad
             4 < p < 6,
        \end{array} \right.
    \end{equation}
where
$$
C_p = \left( \f 2 p \right)^{\f 2 {6-p}} \left( \frac{p-2}
{4 \int_0^1 (1 - s^2)^{\f {4-p}{p-2}}ds}
\right)^{\f {p-2}{6-p}}
$$
\end{corollary}
\noindent
\begin{remark}
\label{rem:al-p}
The threshold $\alpha_p$ is independent of $\rho$ and $\beta$ and is sharp for the problem on the half-line. Indeed, \cite[Proposition 1.4]{BC-23} guarantees that, if \eqref{eq:alpha-mu} is not satisfied, i.e. 
\begin{equation}
\label{eq:alpha-mu2}
\begin{array}{c}
\displaystyle \alpha>\alpha_p(\mu) \\[.1cm]
\text{or} \\[.1cm]
\displaystyle \alpha=\alpha_p(\mu) \quad\text{and}\quad 2<p\leq 4,
\end{array}
\end{equation}
then ground states for $E_\alpha(\cdot,\Rp)$ at mass $\mu$ do not exist and the level $E_{NLS}(\varphi_\mu,\R)$ is not reached.
\end{remark}
\medskip

On the other hand, as the following corollary illustrates, for fixed $\mu$ it is sufficient to make either $\rho$ small or $\beta$ large enough to obtain the existence of a ground state, also when \eqref{eq:alpha-mu} is not satisfied.

\begin{corollary}[Existence for small $\rho$ or large $\beta$] \label{limits}
For every $\mu > 0$ there exist $\widetilde\rho\in \R$ and $\widetilde\beta\geq 0$ such that, if $\rho \leq \widetilde\rho$ or $\beta \geq \widetilde\beta$, then there is a ground state at mass $\mu$ for the energy \eqref{e}. 
\end{corollary}

Note that the existence of $\widetilde\rho$ does not depend on the other parameters of the problem, while its value may depend on them in general. Indeed, if one fixes a state $U = (u,v)\in\D_\mu$ with $q\neq0$, then
\[
\lim_{\rho\to-\infty}E(U)=-\infty.
\]
Therefore, there is a value $\widetilde\rho$ for which $E(U) \leq E_{NLS} (\varphi_\mu, \R)$ for every $\rho\leq\widetilde\rho$, entailing the existence of ground states by Theorem \ref{thm:1}. An analogous result can be established for $\widetilde\beta$ just assuming $q$ and $u(0)$ positive and letting $\beta\to+\infty$.

Now, we focus on nonexistence. To such aim, according to Theorem \ref{thm:1}, one has to prove that the energy of any state exceeds that of the one-dimensional soliton with the same mass. This is guaranteed if $\alpha$ and $\rho$ are large and $\beta$ is small. 

\begin{theorem}[Nonexistence]
\label{thm:2}
Let $p\in(2,6),\,r\in(2,4)\setminus\{r^\star\},\,\mu>0$ . Assume also that \eqref{eq:r-rstar-mu-range} does not hold and \eqref{eq:alpha-mu2} holds. There exists $\rho^\star(r,p,\mu)\in\R$ such that:
\begin{enumerate}[label=(\roman*)]
\item if $\beta=0$, then ground states at mass $\mu$ for the energy \eqref{e} do not exist if and only if $\rho>\rho^\star$;
\item if $\beta>0$ and $\alpha\neq \alpha_p(\mu)$, setting
$$k^\star := \frac{\beta^2}{\alpha - \alpha_p (\mu)},$$
then ground states at mass $\mu$ for the energy \eqref{e} exist when $\rho\leq \rho^\star$ and do not exist when $\rho > \rho^\star + k^\star$ or $\rho=\rho^\star+k^\star$ and $p\in(2,4]$.
\end{enumerate}
\end{theorem}

 Notice that when $\beta>0$ the lower threshold $\rho^\star+k^\star$ for the nonexistence may be larger than the analogous threshold in the case $\beta = 0$, namely $\rho^\star$. 
 Indeed, the presence of $\beta$ may favour the existence of a ground state. To understand this fact, consider the family of states $U_\theta = (u, e^{i \theta} v)$. If
$\beta = 0$, then all such states have the same energy. If $\beta > 0$, then the energy of $U_\theta$ presents an additional term $- \beta \Re (e^{i \theta} q {\overline {u (0)}})$, that can be minimized by choosing the phase $\theta$ so to make 
$e^{i \theta} q {\overline {u (0)}}$ real and positive. This implies that a positive $\beta$ may lower the infimum $\mathcal E (\mu)$ and possibly push it below the threshold given in \eqref{excrit} and then entail existence. In order to compensate this effect and restore nonexistence, the parameter $\rho$ must be pushed forward, and $k^\star$ is a quantitative measure of this need. On the other hand, we are not able to prove that the expression given for $k^\star$ is optimal or to rule out situations for which the optimal $k^\star$ is zero (while Corollary \ref{limits} entails that at least in some cases the optimal $k^\star$ has to be positive). Note also that $k^\star \to0$, as $\beta\to0$. This shows that, as the coupling weakens, the phenomenology converges to that of the decoupled case. On the contrary, $k^\star\to+\infty$, as $\beta\to+\infty$, which suggests that, as the coupling strengthens, it becomes harder not to have ground states. However, at the moment, it remains unclear how to go beyond the threshold $\rho^\star$ for the proof of the existence for any fixed $\beta>0$ (while for $\beta$ large enough the answer is again provided by Corollary \ref{limits}).

The last results illustrate some features of ground states, provided they exist. First, in the case $\beta = 0$ the strict convexity of the function $\mathcal E$ gives complete segregation: the ground state is supported either on the plane, or on the half-line. We omit the proof as it is a straightforward consequence of Step 1 of the proof of Theorem \ref{thm:2} and \cite{BC-23,ABCT-22}.

\begin{theorem}[Shape of ground states. Decoupled case]
\label{thm:3}
 Let $\beta = 0$ and assume that there exists a ground state at mass $\mu$ for the energy \eqref{e}, denoted by $U=(u,v)$. Then, up to a multiplication by a constant phase, the following alternative holds.

\begin{enumerate}[label=(\roman*)]
\item If there is a ground state $\varphi_{\alpha, \mu}$ at mass $\mu$ for $E_{\alpha} (\cdot, \R^+)$ and
\begin{equation} \label{comparisonsoliton}
E_\alpha (\varphi_{\alpha, \mu}, \R^+) \leq  E_\rho (v_{\rho, \mu}, \R^2), 
\end{equation}
with $v_{\rho, \mu}$ a ground state at mass $\mu$ for $E_{\rho} (\cdot, \R^2)$, then $U = (\varphi_{\alpha, \mu}, 0)$ when the inequality is strict, while either $U = (\varphi_{\alpha, \mu}, 0)$ or $U=(0,v_{\rho, \mu})$ when the equality holds. Moreover, the ground state energy is given by
$$
\mathcal E (\mu)  =  E_\alpha (\varphi_{\alpha, \mu}, \R^+).
$$

\item In all other cases, $U = (0, v_{\rho, \mu})$ and the ground state energy reads
$$
\mathcal E (\mu)  =  E_\rho (v_{\rho, \mu}, \R^2).
$$
\end{enumerate}
\end{theorem}

\begin{remark}
The shape of the ground states in such decoupled case can be easily deduced by Theorem \ref{thm:3}. It is indeed well-known (\cite{BC-23,FOO-08,FJ-08}), that a positive ground state for $E_\alpha (\cdot, \R^+)$ is a one-dimensional soliton, translated in such a way that its mass on the half-line equals $\mu$ and Robin boundary condition $u' (0^+) = \alpha u (0)$ is satisfied. Furthermore, by \cite{ABCT-22} we know that, for every ground state $v = \phi + q K_0 / 2 \pi$ for $E (\cdot, \R^2)$, the regular part $\phi$ belongs to $H^2 (\R^2)$ and  the matching condition $\phi (0) = (\rho + (\gamma - \log 2)/2 \pi)q$ is satisfied. Moreover, $v$ is positive, radially symmetric and decreasing along the radial direction, up to a multiplication by a constant phase.
\end{remark}

The last result  treats the general coupled case $\beta>0$.

\begin{theorem}[Shape of ground states. Coupled case]
\label{thm:4}
 Let $\beta > 0$ and assume that there exists a ground state at mass $\mu$ for the energy \eqref{e}, denoted by $U=(u,v)$. Then, $u\neq 0$ and $v =\phi + q K_0 /2 \pi$ with nonvanishing  $q$ and $\phi$. Besides, $\phi \in H^2 (\R^2)$ and the boundary conditions
\begin{equation} \label{bc}
\begin{cases}
u' (0^+)&= \alpha u (0)  -  \beta q  \\[.2cm]
\phi ({\bf 0})&= - \beta u (0)  +  \left( \rho + \frac{\gamma - \log 2}{ 2 \pi} \right) q
\end{cases}
\end{equation}
are satisfied. Moreover, up to a multiplication by a constant phase,
$u$ is a truncated one-dimensional soliton and
$v$ is positive, radially symmetric and decreasing along the radial direction. Finally, the ground state energy  fulfils
$$ \Eps(\mu) < - \theta_p \mu^{\frac{p+2}{6-p}}. $$
\end{theorem}

We remark that the conditions \eqref{bc} are those found by Exner and $\check{\text{S}}$eba in \cite{ES-87} (see Appendix \ref{sec:operatorES}).

Furthermore, while Theorem \ref{thm:3} is an immediate consequence of the concavity properties of the function $\mathcal E$ at $\beta=0$ (\cite{ABCT-22,BC-23,FGI-22}), the proof of Theorem \ref{thm:4} requires some non elementary extensions of the techniques developed in \cite{ABCT-22}, since establishing the features of the part supported on the plane presents some further technical problems.

In addition, in Theorems \ref{thm:3} and \ref{thm:4} the assumption of subcritical nonlinearity $p\in(2,6)$, $r\in(2,4)$ plays no role. Indeed, the shape of $u$ depends on the fact that it satisfies the stationary NLS, whereas the shape of $v$ is obtained by proving that $U=(u,v)$ is a minimizer of the action functional on the corresponding Nehari manifold, that does not require subcriticality (see Section \ref{sec-properties}).

Finally, we mention that the features of the ground states established by Theorems \ref{thm:3} and \ref{thm:4} imply that they belong to the domain of the self-adjoint operator associated with the quadratic part of \eqref{e} (see Appendix \ref{sec:operatorES}).


\subsection{Notation}
We end the present Section by fixing some notation that will be used throughout the paper.

First, we recall the structure of the functions in the energy space $\V$ on the plane (see \cite{ABCT-22,T-23})). Given an arbitrary positive number $\lambda$, any $v \in \V$ can be decomposed as follows
 \begin{equation}
     \label{dec-lambda}
     v : = \phi_\lambda + q \G_\lambda,
 \end{equation}
where $\G_\lambda = \frac 1 {2 \pi} K_0
(\sqrt \lambda \, \cdot  )$. Note that, if $v\in H^1(\R^2)$, then $\phi_\lambda=v$ for every $\lambda>0$; whereas, if $q \neq 0$ and $\nu\neq\lambda$, then
\begin{equation}
\label{change}
\phi_\nu=\phi_\lambda+q(\G_\lambda-\G_\nu).
\end{equation}
Obviously, one recovers the decomposition \eqref{decv} by setting $\lambda = 1$ and $\phi_1 = \phi$. Consistently, the planar component $E_\rho ( \cdot, \R^2)$ of the energy $E$ given in \eqref{e2} rewrites as (see again \cite{ABCT-22})

\begin{equation} 
\label{energy-lambda}
\begin{split}
     E_\rho (v, \R^2)  =  & \frac 1 2
    \| \nabla \phi_\lambda \|^2_{L^2 (\R^2) } + \frac \lambda 2 \left(
    \| \phi_\lambda \|_{L^2 (\R^2)}^2 -
    \| v \|_{L^2 (\R^2)}^2 \right) \\[.2cm]
    & + \left( \rho + \frac{ \gamma - \log 2 + \log \sqrt{\lambda}}{2 \pi} \right) \frac {|q|^2} 2
    - \f1r\| v \|_{L^r (\R^2)}^r 
    \end{split}
\end{equation}

Second, we denote by $Q_\alpha (\cdot, \R^+)$ and $Q_\rho (\cdot, \R^2)$, the quadratic parts of the energies of the linear and the planar components, i.e.
\begin{align}
\label{quadratic1}
Q_\alpha (u, \R^+)  = & \, \| u' \|^2_{L^2 (\R^+)} + \alpha| u (0) |^2 \\[.4cm]
Q_\rho (v, \R^2)  =  & \,
    \| \nabla \phi_\lambda \|^2_{L^2 (\R^2) } +  \lambda  \left(
    \| \phi_\lambda \|_{L^2 (\R^2)}^2 -
    \| v \|_{L^2 (\R^2)}^2 \right) \nonumber\\[.2cm]
    \label{quadratic2}
    & \, + \left( \rho + \frac{ \gamma - \log 2 + \log \sqrt{\lambda}}{2 \pi} \right) |q|^2
\end{align}
so that the energy writes
\begin{equation}
    \label{energyq}
    E (U) = \frac12 Q_\alpha (u, \R^+) +\frac12 Q_\rho (v, \R^2) - \beta \Re (q \overline{u (0)}) - \frac 1 p \| u \|_{L^p (\R^+)}^p - \frac 1 r
    \| v \|_{L^r (\R^2)}^r.
\end{equation}


\section{Necessary and sufficient condition: proof of Theorem \ref{thm:1}}
\label{sec-proof1}

Here we present the proof of the necessary and sufficient condition for the existence of ground states.

\begin{proof}[Proof of Theorem \ref{thm:1}]
The proof that \eqref{excrit} is a necessary condition for existence is straightforward. If a ground state $U$ exists,
then
by point $(1)$ in remark \ref{rem:apriori} it must be
$E(U)\leq -\theta_p \mu^{\frac{p+2}{6-p}}$.

Let us prove the inverse implication. Assume that there exists $W\in \D_\mu$ such that $E(W)\leq -\theta_p\mu^{\frac{p+2}{6-p}}$. If $\Eps(\mu)=E(W)$, then $W$ is a ground state. Suppose instead that 
\begin{equation}
\label{Epsmu<thetap}
\Eps(\mu)<E(W)\leq -\theta_p\mu^{\frac{p+2}{6-p}}
\end{equation}
and consider a minimizing sequence $U_n = (u_n, v_n)$ at a given mass $\mu > 0$ for $E$. To get a lower bound we first use the one-dimensional Gagliardo-Nirenberg estimates
\begin{equation} \label{gn1}
\| u_n \|_{L^p (\R^+)}^p  \leq  C \| u_n \|_{L^2 (\R^+)}^{\frac p 2 + 1} \| u'_n \|_{L^2 (\R^+)}^{\frac p 2 - 1}.
\end{equation}
and
\begin{equation}\label{gn1-inf}
    \|u_n\|_{L^\infty (\Rp)}^2 \leq  C\|u_n\|_{L^2 (\R^+)}\|u_n'\|_{L^2 (\R^+)}.
\end{equation}
Furthermore, to estimate the contribution of the two-dimensional component $v_n$ we use decomposition \eqref{dec-lambda} and proceed by modifying the analogous estimate given in  \cite{ABCT-22} in the following way:
\begin{equation}
\label{pre-gn2}
\begin{split}
\| v_n \|_{L^r (\R^2)}^r & \, \leq   C \| \phi_{\lambda,n} \|_{L^r (\R^2)}^r
+ C \frac{| q_n|^r} \lambda  \\
&\, \leq  C \| \nabla \phi_{\lambda,n} \|_{L^2 (\R^2)}^{r-2} \| \phi_{\lambda,n} \|_{L^2 (\R^2)}^2 + C \frac{| q_n|^r} \lambda  \\
&\, \leq  C \| \nabla \phi_{\lambda,n} \|_{L^2 (\R^2)}^{r-2} 
\| v_n \|_{L^2 (\R^2)}^2 + C \| \nabla \phi_{\lambda,n} \|_{L^2 (\R^2)}^{r-2} 
\frac{| q_n|^2} \lambda  + C \frac{| q_n|^r} \lambda 
\end{split}
\end{equation}
where we used the triangular inequality, the identity $\| G_\lambda \|_{L^s (\R^2)}^s = \frac C \lambda$, and the two-dimensional Gagliardo-Nirenberg estimate
\begin{equation} 
\label{gn2}
\| \phi_{\lambda,n} \|_{L^r (\R^2)}^r  \leq  C  \| \nabla \phi_{\lambda,n} \|_{L^2 (\R^2)}^{r-2}\| \phi_{\lambda,n} \|_{L^2 (\R^2)}^2.
\end{equation}
Now, if $q_n \neq 0$, then choose $\lambda_n = |q_n|^2$, so \eqref{pre-gn2} yields
\begin{equation} \label{gn2gen}
\| v_n \|_{L^r (\R^2)}^r  \leq  C (\| \nabla \phi_{|q_n|^2,n} \|_{L^2 (\R^2)}^{r-2} 
\| v_n \|_{L^2 (\R^2)}^2 +  \| \nabla \phi_{|q_n|^2,n} \|_{L^2 (\R^2)}^{r-2} 
 +  {| q_n|^{r-2}}).
\end{equation}
Let us stress that inequality \eqref{gn2gen}
covers the case $q_n = 0$ too, since in that case $\phi_{|q_n|^2, n} = v_n$ and \eqref{gn2gen} reduces to the ordinary two-dimensional Gagliardo-Nirenberg inequality \eqref{gn2}.
Then, denoting $\phi_{0,n} := v_n$, estimate \eqref{gn2gen} holds for every $v_n$.

Absorbing into the constant $C$ the quantities $\| u_n \|_{L^2 (\R^+)}^2$ and 
$\| v_n \|_{L^2 (\R^2)}^2$, both bounded by $\mu$, controlling the coupling term by using \eqref{gn1-inf} as
$$ - \beta \Re (q_n \overline{u_n (0)})  \geq  - C |q_n|
\| u' \|_{L^2 (\R^+)}^{\f 1 2} \| u \|_{L^2 (\R^+)}^{\f 1 2}
 \geq  - C |q_n |^2 - C \| u' \|_{L^2 (\R^+)}
$$
and using \eqref{energy-lambda}, one obtains
\begin{equation*}
\begin{split}
    E (U_n)  \geq & \, \frac{1}{2}\|u_n'\|_{L^2 (\R^+)}^2-C \left(\| u_n' \|_{L^2 (\R^+)}^{\frac p 2 -1}+ \| u' \|_{L^2 (\R^+)}\right)   \\[.2cm]   & \, + 
    \frac{1}{2}\| \nabla \phi_{|q_n|^2,n} \|^2_{L^2 (\R^2)} -C
    \| \nabla \phi_{|q_n|^2,n} \|^{r-2}_{L^2 (\R^2)}
 + \frac{|q_n|^2}{2} (\log |q_n | - C) -\frac{|q_n|^{r-2}}{r}
  \\[.2cm]
 =:  & \,
f \left(\| u'_n \|_{L^2 (\R^+)}\right) 
+ g \left( \| \nabla \phi_{|q_n|^2,n} \|_{L^2 (\R^2)} \right) 
+ h \left( |q_n|\right)
\end{split}
\end{equation*}
where the one-variable functions $f,g,h$ are lower bounded and divergent with their arguments as $p < 6$ and $r < 4$.  Since $E (U_n)$ converges to $\mathcal E (\mu)$, it is a bounded sequence and,  the quantities $\| u'_n \|_{L^2 (\R^+)}, \, \| \nabla \phi_{|q_n|^2,n} \|_{L^2 (\R^2)}, \, q_n $ must be bounded too. Furthermore, $\| u_n \|_{L^2 (\R^+)}^2 \leq \mu$ and 
$$\| \phi_{|q_n|^2, n} \|_{L^2 (\R^2)}  \leq  \| v_n \|_{L^2 (\R^2)} + |q_n| \| \G_{|q_n|^2} \|_{L^2 (\R^2)}  =  \sqrt{\mu} + C.$$
As a consequence, $u_n$ and $\phi_{|q_n|^2,n}$ are bounded respectively in $H^1 (\R^+)$ and $H^1 (\R^2)$, so that, up to subsequences,
\begin{itemize}
    \item $u_n$ converges to $u$ in the weak topology of $H^1 (\R^+)$,
    \item $q_n$ converges to $q$ in $\mathbb C$,
    \item $\phi_{|q_n|^2,n}$ converges to $\phi^\star$ in the weak topology of $H^1 (\R^2)$.
   \end{itemize}
Moreover, if $q_n=0$, then $\phi_{1,n}=\phi_{|q_n|^2,n}=v_n$, whereas, by \eqref{change}, if $q_n\neq0$, then
$$ \phi_{1,n}  =  \phi_{|q_n|^2,n} + q_n ( \G_{|q_n|^2} - \G_1).$$
Since in this latter case $q_n ( \G_{|q_n|^2} - \G_1)$ admits a weak limit in $H^1 (\R^2)$ it is also bounded in $H^1 (\R^2)$. Thus, $\phi_{1,n}$ is bounded in $H^1 (\R^2)$ too and, again up to subsequences,
\begin{itemize}
    \item $\phi_{1,n}$ converges to $\phi_1$ in the weak topology of $H^1 (\R^2)$ (which is equal to $\phi^\star$ whenever $q=0$).
   \end{itemize}
In addition, $q_n \G_{|q_n|^2}$ converges to $q \G_{|q|^2}$ strongly in $L^s (\R^2)$, if $q \neq 0$, and to $0$ weakly in $L^s (\R^2)$, if $q = 0$, for $s\in[1,+\infty)$. Hence,
\begin{itemize}
    \item $v_n$ converges to $v : = \begin{cases}\phi_{1} + q \G_{1} & \text{if }q\neq0,\\[.2cm]\phi_{1} & \text{if }q=0,\end{cases}$ weakly in $L^s (\R^2)$, for $s\in[2,+\infty)$.
\end{itemize}

Now, we focus on the function $U \in (u,v)$, which is the weak limit in $L^2 (\I)$ of the sequence $U_n$, and call $m$ its mass. Suppose, by contradiction, that $m=0$, then $u = 0$ and $v = 0$. This entails $|u_{n}(0)|\to 0$, as the weak convergence of $u$ holds in $H^1(\R^+)$, and $q_{n}\to 0$. Thus, $\|v_n\|_{L^s(\R^2)}=\|\phi_{1,n}\|_{L^s(\R^2)}+o(1)$, for $s\in[2,+\infty)$, so that
\begin{equation}
\label{breakfree}
\begin{split}
    E (U_n) & =  E_\alpha (u_n, \R^+) + E_\rho (v_n, \R^2) - \beta \Re (q_n
\overline{u_n (0)}
) \\
& =  E_{NLS} (u_n, \R^+) + E_{NLS} (\phi_{1,n}, \R^2) + o (1). 
\end{split}
\end{equation}
In addition, if we define
\begin{equation*}
u_{n}^\star(x):=
\begin{cases}
0\quad &\text{if}\quad x<0,\\
x \quad &\text{if}\quad 0\leq x<|u_{n}(0)|,\\
|u_{n}(x-|u_{n}(0)|)|\quad &\text{if}\quad x\geq |u_{n}(0)|,
\end{cases}
\end{equation*} 
then
\begin{equation} 
\label{unstar}
E_{NLS}(u_{n},\Rp)\geq E_{NLS}(u_{n}^\star,\R)-
 \f 1 2 | u_n (0) | - \f {|u_n (0)|^{p+1}}
{p (p+1)}=E_{NLS}(u_{n}^\star,\R)+o(1).
\end{equation}
Therefore, using \eqref{standard1}, \eqref{standard2}, \eqref{Epsmu<thetap}, \eqref{breakfree} and \eqref{unstar}
\begin{align}
\label{double}
- \theta_p \mu^{\frac{p+2}{6-p}}> &  \,\mathcal E (\mu)  \geq
 \lim_{n\to+\infty} E (U_{n})  =  \lim_{n\to+\infty} \left(E_{NLS}(u_{n}^\star,\R)+E_{NLS}(\phi_{1,n},\Rd)\right)\\
\geq & \lim_{n\to+\infty}
\left( - \theta_p \| u_n^\star \|_{L^2(\R^+)}^{\frac{2p+4}{6-p}} - \tau_r \| \phi_{1,n} \|_{L^2(\R^2)}^{\frac{4}{4-r}} \right) \\
= & \lim_{n\to+\infty} \left( - \theta_p \left(\|u_{n}\|_{L^{2}(\Rp)}^{2}+\f{1}{3}|u_{n}(0)|^{3}
\right)^{\frac{p+2}{6-p}} - \tau_r \| v_n \|_{L^2(\R^2)}^{\frac{4}{4-r}}
\right) \\
= & \lim_{n\to+\infty} \left(- \theta_p \|u_{n}\|_{L^{2}(\Rp)}^{\frac{2p+4}{6-p}} - \tau_r \left( \mu - \| u_n \|_{L^2(\R^+)}^2\right)^{\frac{2}{2-r}} 
\right)
\\
\geq &   \min \left\{ - \theta_p \mu^{\frac{p+2}{6-p}}, - \tau_r \mu^{\frac{2}{4-r}}
\right\},
\end{align}
where in the last line we used the concavity of the functions $- \theta_p \mu^{\frac{p+2}{6-p}}$ and $ - \tau_r \mu^{\frac 2 {4-r}}$. If $\min \left\{ - \theta_p \mu^{\frac{p+2}{6-p}}, - \tau_r \mu^{\frac 2 {4-r}}
\right\}=- \theta_p \mu^{\frac{p+2}{6-p}}$, then the contradiction is immediate. Otherwise, one has
\begin{equation} \label{vanishing}
- \tau_r \mu^{\frac 2 {4-r}}  \leq  \mathcal E (\mu)  <  - \theta_p \mu^{\frac{p+2}{6-p}}.
\end{equation}
Now, as observed in Remark \ref{rem:apriori}, the energy level $-\tau_r \mu^{\frac 2 {4-r}}$ is overcome by the competitor
$\Upsilon_\mu = (0, v_\mu)$, where $v_\mu$ is a ground state at mass $\mu$ for $E_\rho (\cdot, \R^2)$, so 
$$ \mathcal E (\mu)  \, \leq \,  E (\Upsilon_\mu)   < - \tau_r \mu^{\frac 2 {4-r}}  \leq \mathcal E (\mu)$$
which is a contradiction too. As a consequence, $m\neq 0$.

Suppose now $0<m<\mu$. Then 
\begin{equation*}
\|U_{n}-U\|_{L^2(\I)}^{2}=\|u_{n}-u\|_{L^{2}(\Rp)}^{2}+\|v_{n}-v\|_{L^{2}(\Rd)}^{2}=\mu-m+o(1).
\end{equation*}
On the one hand, since $p>2$, $r>2$ and $\f{\mu}{\|U_{n}-U\|_{L^2(\I)}^{2}}>1$, there results
\begin{equation*}
\begin{split}
\mathcal E(\mu)  \leq & \, \mathcal E
\left(\sqrt{\f{\mu}{\|U_{n}-U\|_{L^2(\I)}^{2}}} (U_{n}-U)\right)=\f{1}{2}\f{\mu}{\|U_{n}-U\|_{L^2(\I)}^{2}}Q_{\alpha}(u_{n}-u,\Rp)\\[.2cm]
& \, -\f{1}{p}\left(\f{\mu}{\|U_{n}-U\|_{L^2(\I)}^{2}}\right)^{\f{p}{2}}\|u_{n}-u\|_{L^{p}(\Rp)}^{p}+\f{1}{2}\f{\mu}{\|U_{n}-U\|_{L^2(\I)}^{2}}Q_{\rho}(v_{n}-v,\Rd)\\[.2cm]
& \, -\f{1}{r}\left(\f{\mu}{\|U_{n}-U\|_{L^2(\I)}^{2}}\right)^{\f{r}{2}}\|v_{n}-v\|_{L^{r}(\Rd)}^{r}\\[.2cm]
& \, -\f{\mu}{\|U_{n}-U\|_{L^2(\I)}^{2}}\beta \Re\left((q_{n}-q)\big(\overline{u_{n}(0)}-\overline{u(0)}\big)\right)\\[.2cm]
<& \, \f{\mu}{\|U_{n}-U\|_{L^2(\I)}^{2}} E(U_{n}-U),
\end{split}
\end{equation*}
and, hence,
\begin{equation}
\label{Fsigma-un-u}
\liminf_{n} E (U_{n}-U)\geq \f{\mu-m}{\mu}\mathcal E (\mu).
\end{equation}
On the other hand, by an analogous computation,
\begin{equation*}
\mathcal E (\mu)\leq E \left(\sqrt{\f{\mu}{\|U\|_{L^2(\I)}^{2}}}U\right)<\f{\mu}{\|U\|_{L^2(\I)}^{2}}E(U)
\end{equation*}
and thus
\begin{equation}
\label{Fsigma-u}
E(U)>\f{m}{\mu}\mathcal E (\mu).
\end{equation}
In addition, we can also prove that 
\begin{equation}
\label{Fsigma-BL}
E(U_{n})= E(U_{n}-U)+ E (U)+o(1).
\end{equation}
Indeed, since $u_{n}\deb u$ in $H^{1}(\Rp)$ and $u_{n}(0)\to u(0)$, we have
\begin{equation*}
Q_{\alpha}(u_{n}-u,\Rp)=Q_{\alpha}(u_{n},\Rp)-Q_{\alpha}(u,\Rp)+o(1),
\end{equation*}
and, analogously, since $v_{n}\deb v$ in $L^{2}(\Rd)$, $q_{n}\to q$, and $\phi_{1,n}\deb \phi_{1}$ in $H^{1}(\Rd)$, one has
\begin{equation*}
Q_{\rho}(v_{n}-v,\Rd)=Q_{\rho}(v_{n},\Rd)-Q_{\rho}(v,\Rd)+o(1).
\end{equation*}
Moreover, since, on the one hand, $\|u_{n}\|_{L^{p}(\Rp)}^{p}\leq C$ and $u_{n}\to u$ a.e. on $\Rp$ and, on the other hand, $\|v_{n}\|_{L^{r}(\Rd)}^{r}\leq C$ and $v_{n}\to v$ a.e. on $\Rd$, by the Brezis-Lieb Lemma (\cite{BL-83}) we get
\begin{equation*}
\|u_{n}\|_{L^{p}(\Rp)}^{p}=\|u_{n}-u\|_{L^{p}(\Rp)}^{p}+\|u\|_{L^{p}(\Rp)}^{p}+o(1),
\end{equation*}
and
\begin{equation*}
\|v_{n}\|_{L^{r}(\Rd)}^{r}=\|v_{n}-v\|_{L^{r}(\Rd)}^{r}+\|v\|_{L^{r}(\Rd)}^{r}+o(1).
\end{equation*}
Finally,
\begin{equation*}
\Re\left((q_{n}-q)(\overline{u_{n}(0)}-\overline{u(0)})\right)=\Re\left(q_{n}\overline{u_{n}(0)}\right)-\Re\left(q\overline{u(0)}\right)+o(1),
\end{equation*}
so that, combining with \eqref{Fsigma-un-u}, \eqref{Fsigma-u} and \eqref{Fsigma-BL}, there results that
\begin{equation*}
\mathcal E (\mu)=\liminf_{n} E(U_{n})=\liminf_{n} E(U_{n}-U)+
E (U)>\f{\mu-m}{\mu}\mathcal E (\mu)+\f{m}{\mu}\mathcal E (\mu)=\mathcal E(\mu),
\end{equation*}
which is a contradiction. Thus $m=\mu$ and $U\in \D_{\mu}$, so $u_{n}\to u$ in $L^{2}(\Rp)$, $v_{n}\to v$ in $L^{2}(\Rd)$ and $\phi_{1,n}\to \phi_{1}$ in $L^{2}(\Rd)$.

It is, then, left to prove that
\begin{equation}
\label{Fsigma-lsc}
E(U)\leq \liminf_{n} E(U_{n})=\mathcal E(\mu).
\end{equation}
Because of the previous steps, the proof of \eqref{Fsigma-lsc} reduces to showing that $u_{n}\to u$ in $L^{p}(\Rp)$ and $v_{n}\to v$ in $L^{r}(\Rd)$. On the one hand, by using \eqref{gn1} and the fact that $\|u_{n}'-u'\|_{L^{2}(\Rp)}$ is bounded and $u_{n}\to u$ in $L^{2}(\Rp)$, there results
\begin{equation*}
\|u_{n}-u\|_{L^{p}(\Rp)}^{p}\leq C_{p}\|u_{n}'-u'\|_{L^{2}(\Rp)}^{\f{p}{2}-1}\|u_{n}-u\|_{L^{2}(\Rp)}^{\f{p}{2}+1}\to 0.
\end{equation*}
On the other hand, fixing $\lambda = 1$  in \eqref{pre-gn2}, one has
\begin{equation} 
\|v_{n}-v\|_{L^{r}(\Rd)}^{r}  \leq  
C \left( \| \nabla (\phi_{1,n} -\phi_1) \|_{L^2 (\R^2)}^{r-2} 
(\| v_n - v \|_{L^2 (\R^2)}^2 +
{| q_n - q|^2} ) + {| q_n -q|^r}  \right),
\end{equation}
which goes to $0$ and completes the proof since $\| \nabla \phi_{1,n} \|_{L^2 (\R^2)}$ is bounded, $v_{n}\to v$ in $L^{2}(\Rd)$ and $q_n \to q$.
\end{proof}


\section{Nonexistence: proof of Theorem 2}
\label{sec-exnon}

This section is devoted to the proof of Theorem \ref{thm:2}. Let us first define the quantities
\begin{equation}
\begin{split}
\Eps_\alpha(\mu,\Rp)&:=\inf_{u\in H^1_\mu(\Rp)}E_\alpha(u,\Rp),\\
\Eps_\rho(\mu,\Rd)&:=\inf_{v\in \V_\mu}E_\rho(v,\Rd),
\end{split}
\end{equation}
where the subscript $\mu$ on $H^1(\Rp)$ and on $\V$ denotes that the mass constraint is imposed. Let us also recall that the first eigenvalue of the Laplacian with delta interaction on the plane is given by
\begin{equation}
\label{eq-bottom-2d}
\min_{\substack{v\in \V\\ v\neq0}} \f{Q_{\rho}(v,\Rd)}{\|v\|_{L^{2}(\Rd)}^{2}}=-4e^{-4\pi\rho-2\gamma}=:-\omega_{\rho},
\end{equation}
and the associated eigenspace is spanned by $G_{\omega_\rho}(\x)=\frac{K_0(\sqrt{\omega_\rho}\x)}{2\pi}$ (see \cite[Chapter I.5]{AGHKH-88}).

In the next two lemmas we prove the concavity of $\Eps_{\alpha}(\mu,\Rp)$ and the strict concavity of $\Eps_{\rho}(\mu,\Rd)$ as functions of $\mu$.

\begin{lemma}
\label{Fal-conc}
Let $p\in(2,6)$ and $\alpha\in \R$. Then the function $\Eps_{\alpha}(\cdot,\Rp):[0,+\infty)\to (-\infty, 0]$ is concave and continuous.
\end{lemma}

\begin{proof}
For every fixed $u\in H^{1}_{1}(\Rp)$, define the function
\begin{equation*}
f_{u}(\mu):=E_{\alpha}(\sqrt{\mu}u,\Rp)=\f{\mu}{2}\left(\|u'\|_{L^{2}(\Rp)}^{2}+\alpha|u(0)|^{2}\right)-\f{\mu^{\f{p}{2}}}{p}\|u\|_{L^{p}(\Rp)}^{p}, \qquad \mu\geq0.
\end{equation*}
Since it is concave on $[0,+\infty)$ and
\begin{equation*}
\Eps_{\alpha}(\mu,\Rp)=\inf_{u\in H^{1}_{1}(\Rp)}f_{u}(\mu), \qquad \mu\geq 0,
\end{equation*}
$\Eps_{\alpha}(\cdot,\Rp)$ is concave as well on $[0,+\infty)$. Using the concavity, the facts that $\Eps_{\alpha}(0,\Rp)=0$ and that $\Eps_\alpha(\mu,\Rp)<0$ for every $\mu>0$, it follows that $\Eps_{\alpha}(\cdot,\Rp)$ is also continuous on $[0,+\infty)$.
\end{proof}

\begin{lemma}
\label{F-sig-strict-conc}
Let $r\in (2,4)$ and $\rho\in \R$. Then the function $\Eps_{\rho}(\cdot,\Rd): [0,+\infty)\to (-\infty, 0]$ is strictly concave and continuous.
\end{lemma}

\begin{proof}
We introduce the notation 
\begin{equation}
\label{U-conc}
V_\mu:=\left\{v\in \V_{\mu}\,:\,\|v\|_{L^r(\Rd)}^{r}\geq (2\sqrt{\pi\omega_\rho})^{r}\|\G_{\omega_\rho}\|_{L^r(\Rd)}^{r}\mu^{\f{r}{2}}\right\}.
\end{equation}
Since $2\sqrt{\pi\omega_\rho\mu}\G_{\omega_\rho}\in V$, the set $V_\mu$ is not empty. Moreover, since ground states at mass $\mu$ for $E_\rho(\cdot,\Rd)$ have nonzero regular and singular part \cite[Theorem 1.5]{ABCT-22}, it is easily seen that they belong to $V_\mu$. As a consequence
\begin{equation*}
\Eps_{\rho}(\mu,\Rd)=\inf_{v\in V_\mu}E_{\rho}(v,\Rd)=\inf_{v\in V_1}E_{\rho}(\sqrt{\mu}v,\Rd),\qquad \mu\geq 0,
\end{equation*}
and, since for every fixed $v\in V_1$
the function
\begin{equation*}
f_{v}(\mu):=E_{\rho}(\sqrt{\mu}v,\Rd)=\f{\mu}{2}Q_{\rho}(v,\Rd)-\f{\mu^{\f{r}{2}}}{r}\|v\|_{L^r(\Rd)}^{r}, \qquad \mu\geq 0
\end{equation*} 
is concave on $[0,+\infty)$, $\Eps_{\rho}(\cdot,\Rd)$ is concave too on $[0,+\infty)$. Moreover, by concavity and using the facts that $\Eps_{\rho}(0,\Rd)=0$ and that $\Eps_{\rho}(\mu,\Rd)<0$ for every $\mu>0$, $\Eps_{\rho}(\cdot,\Rd)$ is also continuous on $[0,+\infty)$. Finally, by \eqref{U-conc} we have that, for every fixed $v\in V_1$,
\begin{equation*}
f''_{v}(\mu)=-\f{r-2}{4\mu^{\f{4-r}{2}}}\|v\|_{L^r(\Rd)}^{r}\leq -\f{(r-2)(2\sqrt{\pi\omega_\rho})^{\f{r}{2}}\|\G_{\omega_\rho}\|_{L^r(\Rd)}^{r}}{4\mu^{\f{4-r}{2}}}.
\end{equation*}
Hence the strict concavity of $f_{v}$ is uniform in $v$ on every interval $[a,b]\subset \Rp$, so that $\Eps_{\rho}(\cdot,\Rd)$ is strictly concave on $[0,+\infty)$.
\end{proof}

The next lemma establishes some qualitative properties of $\Eps_{\rho}(\mu,\Rd)$ when $\rho\in\R$ varies and $\mu>0$ is fixed.

\begin{lemma}
\label{lem:Erho}
Let $r\in (2,4)$ and $\mu>0$. Then $\Eps_{\rho}(\mu,\Rd)$ is a strictly increasing and continuous function of $\rho\in\R$ and, denoted by $\xi_\mu$ the only positive and radial ground state for $E_{NLS}(\cdot,\Rd)$, the following limit holds:
\begin{equation*}
\lim_{\rho\to+\infty} \Eps_{\rho}(\mu,\Rd)=E_{NLS}(\xi_\mu,\Rd).
\end{equation*}
\end{lemma}

\begin{proof}
For any $\rho\in\R$, let $v_\rho$ a ground state for $E_{\rho}(\cdot,\Rd)$ at mass $\mu$. In particular, $v_\rho$ is positive up to a multiplication by a constant phase, and, for any  $\la>0$, it can be decomposed as $v_\rho=\phi_{\rho,\la}+q_{\rho}\G_{\la}$ with $q_{\rho}>0$ and $\phi_{\rho,\lambda}\not\equiv0$. Moreover, if we consider the decomposition corresponding to $\la_{\rho}:=q_{\rho}^{2}$, then, by applying \eqref{gn2gen} to \eqref{energy-lambda} with $\la=\la_\rho$ and using that $E_{\rho}(v_\rho,\R^{2})<0$, we get
\begin{equation}
\label{q-sig-bound}
0>E_{\rho}(v_\rho,\R^{2})
\geq \f{1}{2}\|\na \phi_{\rho,\la_{\rho}}\|_{L^2(\Rd)}^{2}-C\|\na \phi_{\rho,\la_{\rho}}\|_{L^2(\Rd)}^{r-2}+ \f{q_{\rho}^{2}}{2}\left(\log(q_\rho)-C\right)-\f{q_\rho^{r-2}}{r},
\end{equation}
which implies that $(q_{\rho})_{\rho\geq c}$ is bounded for every  $c\in\R$, as $r\in(2,4)$.

Now, to prove the full statement of the lemma we divide the proof into four steps.

\emph{Step 1: monotonicity and continuity of $\Eps_{\rho}(\mu,\Rd)$.} Let $\rho_{1}<\rho_{2}$. Since $q_{\rho_{2}}>0$, then
\begin{equation*}
\Eps_{\rho_{1}}(\mu,\Rd)\leq E_{\rho_{1}}(v_{\rho_{2}},\Rd)=E_{\rho_{2}}(v_{\rho_{2}},\Rd)+\f{\rho_{1}-\rho_{2}}{2}q_{\rho_{2}}^{2}<\Eps_{\rho_{2}}(\mu,\Rd),
\end{equation*}
so that $\Eps_{\rho}(\mu,\Rd)$ is strictly increasing in $\rho$. Fix $\widetilde{\rho}\in \R$ and let $\rho>\widetilde{\rho}$. Therefore, using the boundedness of $(q_{\rho})_{\rho\geq c}$ for any $c$,
\begin{equation*}
0<\Eps_{\rho}(\mu,\Rd)-\Eps_{\widetilde{\rho}}(\mu,\Rd)\leq E_{\rho}(v_{\widetilde{\rho}},\Rd)-E_{\widetilde{\rho}}(v_{\widetilde{\rho}},\Rd)=\f{\rho-\widetilde{\rho}}{2}q_{\widetilde{\rho}}^{2}\to 0,\qquad\text{as}\quad \rho\to \widetilde{\rho}^{+},
\end{equation*}
which implies the right continuity at $\widetilde{\rho}$. On the contrary, letting $\rho<\widetilde{\rho}$ and using the boundedness of $(q_{\rho})_{\rho\geq c}$ again,
there results that
\begin{equation*}
0>\Eps_{\rho}(\mu,\Rd)-\Eps_{\widetilde{\rho}}(\mu,\Rd)\geq E_{\rho}(v_\rho,\Rd)-E_{\widetilde{\rho}}(v_\rho,\Rd)=\f{\rho-\widetilde{\rho}}{2}q_{\rho}^{2}\to 0,\qquad\text{as}\quad \rho\to\widetilde{\rho}^{-},
\end{equation*}
which implies left continuity at $\widetilde{\rho}$.

\emph{Step 2: $\lim_{\rho \to +\infty}q_{\rho}=0$.} Assume by contradiction that there exists a positively diverging sequence $(\rho_j)_j\in\R$ such that $q_{\rho_j}\geq C>0$ for every $j$. Omitting the dependence on $j$ in the next lines, if we decompose a ground state for $F_{\rho}(\cdot,\Rd)$ at mass $\mu$ as $v_\rho=\phi_{\la_{\rho}}+q_{\rho}\G_{\la_{\rho}}$ with $\la_{\rho}:=q_{\rho}^{2}$, then \eqref{q-sig-bound} holds and hence, using the fact that $q_\rho$ is bounded from above and from below away from zero, there results
\begin{equation*}
\rho\leq C+C\|\na \phi_{\rho,\la_{\rho}}\|_{L^2(\Rd)}^{r-2}-\frac{1}{2} \|\na \phi_{\rho,\la_{\rho}}\|_{L^2(\Rd)}^2,
\end{equation*}
which contradicts the positive divergence of $(\rho_j)_j$.

\emph{Step 3: $\lim_{\rho \to +\infty}\rho q_{\rho}^{2}=0$.} Let $v_\rho=\phi_{\rho,\la}+q_{\rho}\G_{\la}$ be a ground state at mass $\mu$ for $E_{\rho}(\cdot,\Rd)$. Easy computations yield
\begin{multline}
\label{phila->mu}
\big|\sqrt{\mu}-\|\phi_{\rho,\la}\|_{L^2(\Rd)}\big|=\big|\|v_\rho\|_{L^2(\Rd)}-\|\phi_{\rho,\la}\|_{L^2(\Rd)}\big|\\[.2cm]
\leq\|v_\rho-\phi_{\rho,\lambda}\|_{L^2(\Rd)}= \f{q_{\rho}}{2\sqrt{\pi\la}}\to 0,\qquad \quad\rho \to+\infty,
\end{multline}
so that $\|\phi_{\rho,\la}\|_{L^2(\Rd)}^2\to \mu$ as $\rho\to+\infty$. Moreover, using \eqref{gn2}, \eqref{phila->mu} and $q_{\rho}\to 0$, one obtains 
\begin{equation*}
0>E_{\rho}(v_\rho,\Rd)\geq \f{1}{2}\|\na \phi_{\rho,\la}\|_{L^2(\Rd)}^{2}-C\|\na \phi_{\rho,\la}\|_{L^2(\Rd)}^{r-2}+\f{\rho q_\rho^{2}}{2}+o(1),\qquad\text{as}\quad \rho\to+\infty,
\end{equation*}
and hence $\|\na\phi_{\rho,\la}\|_{L^2(\Rd)}$ is bounded for large $\rho$. As a consequence, by Sobolev embeddings $\|\phi_{\rho,\la}\|_{L^r(\Rd)}$ and $\|v_\rho\|_{L^r(\Rd)}$ are bounded for large $\rho$. Therefore,
\begin{multline}
\label{pnorm->0}
\big|\|v_\rho\|_{L^r(\Rd)}^{r}-\|\phi_{\rho,\la}\|_{L^r(\Rd)}^{r}\big|\\[.2cm]
\leq r\sup\left\{\|v_\rho\|_{L^r(\Rd)}^{r-1}, \|\phi_{\rho,\la}\|_{L^r(\Rd)}^{r-1}\right\}|q_{\rho}|\|\G_{\la}\|_{L^r(\Rd)}\to 0,\qquad\text{as}\quad \rho\to+\infty.
\end{multline}
Finally, exploiting \eqref{standard2}, \eqref{comparison2}, \eqref{phila->mu} and \eqref{pnorm->0} we have
\begin{multline*}
0>\mathcal{E}_\rho(\mu,\Rd)-E_{NLS}(\xi_\mu,\Rd)=\mathcal{E}_\rho(\mu,\Rd)+\tau_r\|\phi_{\rho,\la}\|_{L^2(\Rd)}^\frac{4}{4-r}+o(1)\\[.2cm]
\geq E_{\rho}(v_\rho,\Rd)-E_{NLS}(\phi_{\rho,\la},\Rd)+o(1)\geq \f{\rho q_{\rho}^{2}}{2}+o(1),\qquad\text{as}\quad \rho\to+\infty,
\end{multline*}
so that $\rho q_{\rho}^{2}\to 0$ as $\rho\to+\infty$.

\emph{Step 4: limit for $\rho\to+\infty$.} Arguing as before and using \emph{Step 2} and \emph{Step 3} there results
\begin{multline*}
0<E_{NLS}(\xi_\mu,\Rd)-\Eps_{\rho}(\mu,\Rd)\leq E_{NLS}(\phi_{\rho,\la},\Rd)-E_{\rho}(v_\rho,\Rd)+o(1)\\[.2cm]
\leq -\f{\rho q_{\rho}^{2}}{2}+o(1)\to 0,\qquad\text{as} \quad \rho\to+\infty,
\end{multline*}
which concludes the proof.
\end{proof}

Now we have all the ingredients to prove Theorem \ref{thm:2}.

\begin{proof}[Proof of Theorem \ref{thm:2}]
Preliminarily, if $\beta=0$ then
\begin{equation}
\label{eq:level-beta0}
    \Eps(\mu)=\min\{\Eps_\alpha(\mu,\Rp),\Eps_\rho(\mu,\Rd)\}
\end{equation}
and every possible ground state $U$ for $E$ at mass $\mu$ is either of the form $U=(u,0)$ or $U=(0,v)$. Indeed, since $\beta=0$
\begin{equation*}
E(U)=E_{\alpha}(u,\Rp)+E_{\rho}(v,\Rd),\qquad U=(u,v)\in\D_\mu,
\end{equation*}
and thus
\begin{equation}
\label{F-sig-dec}
\Eps(\mu)=\inf_{m\in [0,\mu]}\left\{\Eps_{\alpha}(m,\Rp)+\Eps_{\rho}(\mu-m,\Rd)\right\}.
\end{equation}

Since $\Eps_{\alpha}(\cdot,\R)$ is concave and continuous in $[0,+\infty)$ by Lemma \ref{Fal-conc}, and $\Eps_{\rho}(\cdot,\Rd)$ is strictly concave and continuous in $[0,+\infty)$ by Lemma \ref{F-sig-strict-conc}, the function
\[
m\mapsto \Eps_{\alpha}(m,\Rp)+\Eps_{\rho}(\mu-m,\Rd)
\]
is strictly concave and continuous in $[0,\mu]$, hence it attains its minimum either at $m=0$ or at $m=\mu$. This entails \eqref{eq:level-beta0} and the fact that either $U=(u,0)$ or $U=(0,v)$.

Now we can prove the two claims of the theorem.

\emph{Proof of (i).}  Since \eqref{eq:alpha-mu2} holds, by Remark \ref{rem:al-p} it is guaranteed that $E_\alpha(u,\Rp)>E_{NLS}(\varphi_\mu,\R)$ for every $u\in H^1(\Rp)$. Thus, by Theorem \ref{thm:1} ground states cannot be of the form $U=(u,0)$. Therefore, when existing, they are of the form $U=(0,v)$.
Denoted by $v_\mu$ a ground state for $E_\rho(\cdot,\Rd)$ at mass $\mu$ and by  $\Upsilon_\mu$ the state  $(0, v_\mu)\in \D_\mu$, by Theorem \ref{thm:1} ground states for $E$ at mass $\mu$ exist if and only if 
\begin{equation*}
E(\Upsilon_\mu)=E_\rho(v_\mu,\Rd)\leq E_{NLS}(\varphi_\mu,\R).
\end{equation*}
Since $\lim_{\rho\to-\infty}\Eps_{\rho}(\mu,\Rd)=-\infty$ and \eqref{eq:r-rstar-mu-range} does not hold, by Lemma \ref{lem:Erho} there exists $\rho^\star\in \R$ such that $E_\rho(v_\mu,\Rd)>E_{NLS}(\varphi_\mu,\R)$ if and only if $\rho>\rho^\star$, entailing the thesis.

\emph{Proof of (ii).} The existence of ground states for $E$ at mass $\mu$ when $\rho\leq \rho^\star$ follows by Theorem \ref{thm:1} arguing as before, since 
\begin{equation*}
    \Eps(\mu)\leq E(\Upsilon_\mu)=E_\rho(v_\mu,\Rd)\leq E_{NLS}(\varphi_\mu,\R),\qquad \forall\, \rho\leq \rho^\star.
\end{equation*}
Fix now $t>0$. By Young's inequality, for every $U=(u,v)\in\D_\mu$
\begin{equation}
    \begin{split}
    \label{eq-newest2}
 E(U) &=  E_\alpha(u,\Rp)+E_\rho(v,\Rd)-\beta\Re\left(q\overline{u(0)}\right)\\
          &\geq  E_\alpha(u,\Rp)+E_\rho(v,\Rd)-\f{\beta}{2t}|u(0)|^2-\f{\beta t}{2}|q|^2\\
             &=  E_{\alpha-\f{\beta}{t}}(u,\Rp)+E_{\rho-\beta t}(v,\Rd)=:\widetilde{E}(U).
             \end{split}
\end{equation}
In particular, since in $\widetilde{E}$ there is no coupling there results that
\begin{equation*}
\Eps(\mu)\geq \inf_{U\in\D_\mu}\widetilde{E}(U)=\min\left\{\Eps_{\alpha-\f{\beta}{t}}(\mu,\Rp), \Eps_{\rho-\beta t}(\mu,\Rd)\right\}.
\end{equation*}
If \eqref{eq:alpha-mu2} holds with $\alpha$ replaced by $\alpha-\f{\beta}{t}$, i.e.
\[
\alpha-\f{\beta}{t}>\alpha_p(\mu)\quad\text{or}\quad \left(\alpha-\f{\beta}{t}=\alpha_p(\mu)\quad\text{and}\quad 2<p\leq 4\right),
\]
then by Remark \ref{rem:al-p} $E_{\alpha-\f{\beta}{t}}(u,\Rp)>E_{NLS}(\varphi_\mu,\R)$ for every $u\in H^1(\Rp)$. Let us observe that the solution of the equation $\alpha-\frac{\beta}{t}=\alpha_p(\mu)$ is given by $t^\star=\frac{\beta}{\alpha-\alpha_p(\mu)}$.

Arguing as in the \emph{Proof of (1)} and denoting with $k^\star:=\beta t^\star$, one deduces that ground states for $\widetilde{E}$ at mass $\mu$ do not exist if $\rho>\rho^\star+k^\star$ or $\rho=\rho^\star+k^\star$ and $2<p\leq 4$, hence by \eqref{eq-newest2} and Theorem \ref{thm:1} ground states for $E$ at mass $\mu$ do not exist.
\end{proof}


\section{Shape of the ground states in the coupled case: proof of Theorem \ref{thm:4}}
\label{sec-properties}

Here we establish the features of the ground states at mass $\mu$ for the energy \eqref{e} in the coupled case $\beta>0$. To do this, some preliminary results are necessary.

First we present the equations and the boundary conditions satisfied by such ground states. 

\begin{lemma}
\label{lem:EL-eq}
Let $\beta> 0$. If $\,U=(u,v)$ is a ground state at mass $\mu$ for the energy \eqref{e}, then $u\in H^2(\Rp)$, $\phi_\la:=v-q\G_\la\in H^2(\Rd)$ for every $\la>0$, and  there exists $\omega^\star=\omega^\star(U)$ such that
\begin{align}
\label{eq:half-line}
&u''+|u|^{p-2}u=\omega^\star u,\\[.2cm]
\label{eq:plane}
&\lap \phi_\la+|v|^{r-2}v=\omega^\star \phi_\la+(\omega^\star-\lambda)q\G_\la,\\[.2cm]
\label{first-cond}
&u'(0)=\alpha u(0)-\beta q,\\[.2cm]
\label{sec-cond}
&\phi_\la(0)=-\beta u(0)+\left(\rho+\f{\gamma-\log(2)+\log(\sqrt{\la})}{2\pi}\right)q.
\end{align}
Moreover, $\omega^\star$ satisfies
\begin{equation}
\label{eq:omega-star}
\omega^\star(U)=\frac{\|u\|_{L^p(\Rp)}^p+\|v\|_{L^r(\Rd)}^r-Q(U)}{\mu}.
\end{equation}
\end{lemma}

\begin{proof}
First, by the Lagrange Multipliers theorem 
\begin{multline}
\label{EL-eq-hyb}
 \scal{\eta'}{u'}_{L^{2}(\Rp)}+\alpha\overline{\eta(0)}u(0)-\scal{\eta}{|u|^{p-2}u}_{L^{2}(\Rp)}+\omega^\star\scal{\eta}{u}_{L^{2}(\Rp)}\\[.2cm]
 +\scal{\na \chi_{\la}}{\na \phi_{\la}}_{L^{2}(\Rd)}+\la \scal{\chi_{\la}}{\phi_{\la}}_{L^{2}(\Rd)}+(\omega^\star-\la)\scal{\chi}{v}_{L^{2}(\Rd)}+\overline{Q}q(\rho+\theta_{\la})\\[.2cm]
 -\scal{\chi}{|v|^{r-2}v}_{L^{2}(\Rd)}-\beta\overline{Q}u(0)-\beta q\overline{\eta(0)}=0,
\end{multline}
for every $\eta\in H^{1}(\Rp)$ and every $\chi=\chi_{\la}+Q\G_{\la}\in \V$. Notice that the expression \eqref{eq:omega-star} can be deduced from \eqref{EL-eq-hyb} by taking $\eta=u$ and $\chi=v$. Now, if $\chi=0$ and $\eta \in H^{1}_{0}(\Rp)$ in \eqref{EL-eq-hyb}, then 
\begin{equation*}
\scal{\eta'}{u'}_{L^{2}(\Rp)}+\scal{\eta}{\omega^\star u-|u|^{p-2}u}_{L^{2}(\Rp)}=0,
\end{equation*}
and thus, since $\omega^\star u-|u|^{p-2}u\in L^{2}(\Rp)$, $u\in H^2(\Rp)$ and \eqref{eq:half-line} holds. On the other hand, if $\eta=0$ and $Q=0$ in \eqref{EL-eq-hyb}, then 
\begin{equation}
\label{EL-plane}
\scal{\na \chi}{\na \phi_{\la}}_{L^{2}(\Rd)}+\scal{\chi}{\omega^\star \phi_{\la}+q(\omega^\star-\la)\G_{\la}-|v|^{r-2}v}_{L^{2}(\Rd)}=0.
\end{equation}
Thus, arguing as before, $\phi_\la\in H^2(\Rd)$ and \eqref{eq:plane} is satisfied. Finally, setting $\chi_{\la}=0$ and $Q=1$ in \eqref{EL-eq-hyb} and applying \eqref{eq:half-line} and \eqref{eq:plane}, there results
\begin{equation}
\label{eq-almostcond}
-\overline{\eta(0)}u'(0)+\alpha \overline{\eta(0)}u(0)-\scal{\G_{\la}}{\left(-\lap+\la\right)\phi_{\la}}+q(\rho+\theta_{\la})-\beta u(0)-\beta q\overline{\varphi(0)}=0
\end{equation}
for every $\eta\in H^{1}(\Rp)$. 
Now, whenever $\eta(0)=0$, since $\scal{\G_{\la}}{\left(-\lap+\la\right)\phi_{\la}}=\phi_{\la}(\z)$, we get \eqref{sec-cond}. On the other hand, setting $\eta(0)=1$ and plugging \eqref{first-cond} in \eqref{eq-almostcond}, \eqref{first-cond} follows. 
\end{proof}

Then, relying on Lemma \ref{lem:EL-eq} we show that the ground states must be supported on both the half-line and the plane.

\begin{proposition}
\label{prop:uv-neq-0}
Let $\beta> 0$. If $U=(u,v)$ is a ground state at mass $\mu$ for the energy \eqref{e}, then $u\not\equiv 0$ and $v\not\equiv 0$. In particular $\phi_{\la}=v-q\G_\la\not\equiv 0$ for every $\lambda>0$, and $q\neq 0$.

In addition, $q>0$ and $u(x)>0$ for every $x\geq 0$ up to multiplication by a constant phase.
\end{proposition}

\begin{proof}
Preliminarily, note that if $U=(u,v)$ is a ground state at mass $\mu$ for the energy \eqref{e}, then $U$ satisfies \eqref{first-cond}-\eqref{sec-cond}. Now, for the sake of clarity, we divide the proof into several steps.

\emph{Step 1: $v\not\equiv0$.} Assume by contradiction  $v\equiv 0$. Therefore, since $\beta\neq 0$, \eqref{sec-cond} implies $u(0)=0$. Thus, denoted with $\varphi_\mu$ the one-dimensional soliton of mass $\mu$, by \eqref{bound1} and the fact that $u(0)=0$
\begin{equation*}
E(U)=E_\alpha(u,\Rp)>E_{NLS}(\varphi_\mu,\R)\geq \Eps(\mu),
\end{equation*}
which is a contradiction.

\emph{Step 2: $u\not\equiv0$.} Assume by contradiction  $u\equiv 0$. Therefore, since $\beta\neq 0$, \eqref{first-cond} implies $q=0$, so that $v\in H^{1}(\Rd)$ and $v(0)=0$. Moreover, by \cite[Theorem 1.5]{ABCT-22} $v$ cannot be a ground state for $E_{NLS}(\cdot,\Rd)$ at mass $\mu$ since it is not positive up to the multiplication by a constant phase. Thus, denoted by $\xi_\mu$ the only positive and spherically symmetric ground state of $E_{NLS}(\cdot,\Rd)$, by \eqref{bound2} there results
\begin{equation*}
E(U)=E_\rho(v,\Rd)=E_{NLS}(v,\Rd)>E_{NLS}(\xi_\mu,\Rd)\geq \Eps(\mu),
\end{equation*}
which is a contradiction.

\emph{Step 3: $q\neq 0$.} Assume by contradiction  $q=0$. This entails that $v\in H^{1}(\Rd)$ and
\begin{equation*}
E(U)=E_{\alpha}(u,\Rp)+E_{NLS}(v,\Rd).
\end{equation*}
As a consequence
\begin{equation*}
\Eps(\mu)=\inf_{m\in[0,\mu]}\left\{\Eps_{\alpha}(m,\Rp)+E_{NLS}(\varphi_{\mu-m},\Rd)\right\}.
\end{equation*}
Since $\Eps_{\alpha}(\cdot,\Rp)$ is a concave and continuous function on $[0,\mu]$, by Lemma \ref{Fal-conc} and 
\[
E_{NLS}(\varphi_{\mu-m},\Rd)=-\tau_r(\mu-m)^{\f{2}{4-r}},
\]
there results that the function
\[
m\mapsto\Eps_{\alpha}(m,\Rp)+\Eps(\mu-m,\Rd)
\]
is strictly concave and continuous on $[0,\mu]$. Thus it attains its minimum at $m=0$ or $m=\mu$, but this would imply either $u\equiv 0$ or $v\equiv 0$, which are prevented by Step 1 and Step 2.

\emph{Step 4: $\phi_{\la}\not\equiv 0$, for every $\la>0$.} Assume by contradiction $\phi_{\la}\equiv 0$ for some $\lambda>0$. Since $v$ solves \eqref{eq:plane}, $q\neq0$ and $\G_\lambda(\x)>0$ for all $\x\in \Rd\setminus\{0\}$, there results 
\begin{equation*}
\omega-\lambda=|q|^{r-2}\G_{\la}^{r-2}(\x),\qquad \x\in \Rd\setminus\{0\},
\end{equation*}
which is false.

\emph{Step 5: $q>0$ and $u(x)>0$ for every $x\geq0$ up to a multiplication by a constant phase.} Assume $q=|q|\mathrm{e}^{i\eta}$ with $\eta\neq 2k\pi$, $k\in\Z$ and define $U_\eta=(\mathrm{e}^{-i\eta}u,\mathrm{e}^{-i\eta}v)=(\mathrm{e}^{-i\eta}u,\mathrm{e}^{-i\eta}\phi_\la+|q|\G_\la)$. Since $\|U_\eta\|_{L^2(\I)}^2=\|U\|_{L^2(\I)}^2=\mu$ and
 \begin{equation*}
 \begin{split}
  E(U_\eta)&=E_\alpha(\mathrm{e}^{-i\eta}u,\Rp)+E_\rho(\mathrm{e}^{-i\eta}\phi_\la+|q|\G_\la,\Rd)-\beta\Re\left(|q|\overline{\mathrm{e}^{-i\eta}u(0)}\right)\\
                  &=E_\alpha(u,\Rp)+E_\rho(v,\Rd)-\beta\Re\left(\mathrm{e}^{i\eta}|q|\overline{u(0)}\right)=E(U),
                  \end{split}
 \end{equation*}
 $U_\eta$ is a ground state for $E$ at mass $\mu$. Hence $q>0$ up to  multiplication by a constant phase. Similarly, since the only $L^2$ solutions of \eqref{eq:half-line} is the one-dimensional soliton up to translation and multiplication by a constant phase and \eqref{first-cond} holds, then either $u>0$ or $u<0$. However, if it were negative then $-\beta\Re\left(q(\overline{-u(0)})\right)<-\beta\Re\left(q\overline{u(0)}\right)$, so that $E(U)>E(-u,v)$, which is absurd. Thus $u$ is positive.
\end{proof}

Finally, in order to prove Theorem \ref{thm:4} it is necessary to exhibit some further variational properties fulfilled by the ground states of the energy \eqref{e}. To this aim we introduce the action functional $S_{\omega}$. Given $\omega\in \R$, the action functional at frequency $\omega$ is the functional $S_{\omega}:\D\to \R$ such that
\begin{equation}
\label{S-om}
S_{\omega}(U):=E(U)+\f{\omega}{2}\|U\|_{L^{2}(\I)}^{2}.
\end{equation}
In addition, the Nehari's manifold at frequency $\omega$ is defined by
\begin{equation}
N_{\omega}:=\{U\in \D\setminus \{0\}\,:\,I_{\omega}(U)=0\},
\end{equation}
where $I_{\omega}:\D\to \R$ is the functional
\begin{equation}
\label{I-om}
I_{\omega}(U):=Q_{\omega}(U)-\|u\|_{L^{p}(\Rp)}^{p}-\|v\|_{L^{r}(\Rd)}^{r}, \qquad U = (u,v)
\end{equation}
with
\begin{equation*}
Q_{\omega}(U):=Q(U)+\omega\|U\|_{L^{2}(\I)}^{2}.
\end{equation*}
Then we give the following definition.

\begin{definition}[Action minimizers]
\label{def:ac}
We call \emph{action minimizer} at frequency $\omega\in\R$ any function $U \in N_{\omega}$ such that
\begin{equation}
S_{\omega}(U)=d(\omega):=\inf_{N_{\omega}}S_{\omega}.
\end{equation}
\end{definition}

\begin{remark}
\label{rem-bs}
 Action minimizers at frequency $\omega$ satisfy \eqref{eq:half-line}, \eqref{eq:plane}, \eqref{first-cond} and \eqref{sec-cond}, with $\omega^\star$ replaced by $\omega$.
\end{remark}

In the following lemma we establish a  connection between ground states at fixed mass for the energy \eqref{e} and action minimizers at fixed frequency. The proof follows the ideas of \cite{DST-23, JL-22}.

\begin{lemma}
\label{gs->min}
If $U$ is a ground state at mass $\mu$ for the energy \eqref{e}, then $U$ is an action minimizer at frequency $\omega^\star$. Moreover
$
\omega^{\star}>E_{\rm{lin}},
$
with $E_{\rm{lin}}$ defined as in \eqref{Elin}.
\end{lemma}

\begin{proof}
Let $U=(u,v)$ be a ground state for $E$ at mass $\mu$ and let $\omega^\star$ be the associated frequency defined by \eqref{eq:omega-star}. Now assume by contradiction that there exists $\widetilde{U}=(\widetilde{u}, \widetilde{v}) \in N_{\omega^\star}$ such that $S_{\omega^\star}(\widetilde{U})<S_{\omega^\star}\left(U\right)$. Fixing $\tau>0$ in such a way that $\|\tau \widetilde{U}\|_{L^{2}(\I)}^2=\mu$, there results
\begin{equation*}
S_{\omega^\star}(\tau \widetilde{U})=\f{\tau^2}{2}Q_{\omega^\star}(U)-\f{\tau^p}{p}\|\widetilde{u}\|_{L^{p}(\Rd)}^p-\f{\tau^{r}}{r}\|\widetilde{v}\|_{L^{r}(\Rd)}^{r}.
\end{equation*}
In addition, since $\widetilde{U}\in N_{\omega^\star}$, by \eqref{I-om} we get
\begin{equation*}
\begin{split}
\f{d}{d\tau}S_{\omega^\star}(\tau \widetilde{U})&=\tau Q_{\omega^\star}(U)-\tau^{p-1}\|\widetilde{u}\|_{L^{p}(\Rd)}^p-\tau^{r-1}\|\widetilde{v}\|_{L^{r}(\Rd)}^{r}\\
&=\tau I_{\omega^\star}(\widetilde{U})+\tau\left[(1-\tau^{p-2})\|\widetilde{u}\|_{L^{p}(\Rp)}^p+(1-\tau^{r-2})\|\widetilde{v}\|_{L^{r}(\Rd)}^{r}\right]\\
&=\tau\left[(1-\tau^{p-2})\|\widetilde{u}\|_{L^{p}(\Rp)}^p+(1-\tau^{r-2})\|\widetilde{v}\|_{L^{r}(\Rd)}^{r}\right],
\end{split}
\end{equation*}
which is greater than or equal to zero if and only if $0<\tau\leq 1$. Hence $S_{\omega^\star}(\tau \widetilde{U})\le S_{\omega^\star}(\widetilde{U})$ for every $\tau>0$. Therefore, since $S_{\omega^\star}(\tau \widetilde{U})\le S_{\omega^\star}(\widetilde{U}) <S_{\omega^\star}(U)$,
\begin{equation*}
E(\tau \widetilde{U})+\f{\omega^\star}{2}\|\tau \widetilde{U}\|_{L^{2}(\I)}^2<E(U)+\f{\omega^\star}{2}\|U\|_{L^{2}(\I)}^2.
\end{equation*}
However, as $\|\tau \widetilde{U}\|_{L^{2}(\I)}^2=\|U\|_{L^{2}(\I)}^2=\mu$, this entails that $E(\tau \widetilde{U})<E(U)$, which is a contradiction. It is then left to estimate $\omega^\star$. Using \eqref{eq:omega-star} and \eqref{bound3}, it follows that
\begin{multline*}
-\omega^\star=\f{Q(U)-\|u\|_{L^{p}(\Rp)}^{p}-\|v\|_{L^{r}(\Rd)}^{r}}{\mu}\\[.2cm]
=\f{2E(U)-\f{p-2}{p}\|u\|_{L^{p}(\Rp)}^{p}-\f{r-2}{r}\|v\|_{L^{r}(\Rd)}^{r}}{\mu}<2\f{\Eps(\mu)}{\mu}<-E_{\rm{lin}}.
\end{multline*}
\end{proof}

As a consequence of Lemma \ref{gs->min}, in order to prove Theorem \ref{thm:4} it is relevant to investigate the features of the action minimizers at frequency $\omega >E_{\rm{lin}}$.

In order to study the features of the action minimizers at frequency $\omega$, we have to introduce two equivalent minimization problems. Preliminarily, we introduce the quantities
\begin{equation}
\widetilde{S}(U):=\f{p-2}{2p}\|u\|_{L^{p}(\Rp)}^{p}+\f{r-2}{2r}\|v\|_{L^{r}(\Rd)}^{r}
\end{equation}
and
\begin{equation}
A_{\omega}(U):=\f{r-2}{2r}Q_{\omega}(U)+\f{p-r}{pr}\|u\|_{L^{p}(\Rp)}^{p}.
\end{equation}
Both functionals are well defined on $\D$ and, using \eqref{I-om}, there results
\begin{equation}
\label{equiv-form-Som}
S_{\omega}(U)=\f{1}{2}I_{\omega}(U)+\widetilde{S}(U)=\f{1}{r}I_{\omega}(U)+A_{\omega}(U),\qquad  U \in\D,
\end{equation}
so that
\begin{equation}
\label{equiv-form-Som2}
 S_{\omega}(U)=\widetilde{S}(U)=A_{\omega}(U),\qquad \, U \in N_\omega
\end{equation}
which implies 
\begin{equation}
\label{equiv-form-d}
d(\omega)=\inf_{N_{\omega}}\widetilde{S}=\inf_{N_{\omega}}A_{\omega}.
\end{equation}
In other words, the problems of the minimization of $S_\omega$, $\widetilde{S}$, and $A_\omega$ are equivalent on $N_\omega$. In the following two lemmas we show that they are  equivalent also when the minimization domains for $\widetilde{S}$ and $A_\omega$ are suitably modified.

\begin{lemma}
\label{lem:dmu-alt}
Let $p,\,r>2$, $\beta>0$ and $\omega>E_{\rm{lin}}$. Then
\begin{equation}
\label{d-om-alt}
d(\omega)=\inf_{\widetilde{N}_{\omega}}\widetilde{S},\qquad\text{with}\qquad\widetilde{N}_{\omega}:=\{U\in \D\setminus\{0\}\,:\,I_{\omega}(U)\leq 0\}.
\end{equation}
Moreover
\begin{equation}
\label{eq-secondequiv}
\left\{\begin{array}{l}\displaystyle\widetilde{S}(U)=d(\omega)\\[.2cm]\displaystyle U\in \widetilde{N}_\omega\end{array}\right.\qquad\Longleftrightarrow\qquad \left\{\begin{array}{l}\displaystyle S_{\omega}(U)=d(\omega)\\[.2cm]U\in N_\omega.\end{array}\right.
\end{equation}
\end{lemma}

\begin{proof}
First we prove \eqref{d-om-alt}. On the one hand, since $N_\omega\subset\widetilde{N}_\omega$, combining \eqref{equiv-form-Som2} and \eqref{equiv-form-d} there results
\begin{equation*}
d(\omega)\geq\inf_{\widetilde{N}_{\omega}}\widetilde{S}.
\end{equation*}
On the other hand, let $U\in \D\setminus\{0\}$ satisfy $I_{\omega}(U)<0$. For $\tau>0$
\begin{equation*}
I_{\omega}(\tau U)=\tau^{2}Q_{\omega}(U)-\tau^{p}\|u\|_{L^{p}(\Rp)}^{p}-\tau^{r}\|v\|_{L^{r}(\Rd)}^{r}.
\end{equation*}  
In particular, since $p,\,r>2$ and since $\omega>E_{\rm{lin}}$ implies $Q_{\omega}(U)>0$, there exists $\tau^\star\in (0,1)$ such that $I_{\omega}(\tau^\star U)=0$. In addition
\begin{equation*}
\widetilde{S}(\tau^\star U)<\widetilde{S}(U)
\end{equation*}
and thus, by \eqref{equiv-form-d}
\begin{equation*}
d(\omega)\leq \inf_{\widetilde{N}_{\omega}}\widetilde{S}.
\end{equation*}

Finally we prove \eqref{eq-secondequiv}. By \eqref{equiv-form-Som2}, if $U\in N_{\omega}$ and $S_{\omega}(U)=d(\omega)$, then $U\in \widetilde{N}_{\omega}$ and $\widetilde{S}(U)=d(\omega)$. Moreover, if $U\in \widetilde{N}_{\omega}\setminus N_{\omega}$ and $\widetilde{S}(U)=d(\omega)$ then, arguing as before, there exists $\tau^\star\in (0,1)$ such that $\tau^\star U\in N_{\omega}$ and $S_{\omega}(\tau^\star U)<d(\omega)$, which contradicts \eqref{d-om-alt}. Therefore, if $U\in\widetilde{N}(\omega)$ and $\widetilde{S}(U)=d(\omega)$, then $U\in N_{\omega}$ and $S_{\omega}(U)=d(\omega)$.
\end{proof}

\begin{lemma}
\label{lem:thirdform}
Let $p,\,r>2$, $\beta>0$ and $\omega>E_{\rm{lin}}$.
Then,
\begin{equation}
\label{eq-levA}
d(\omega)=\inf_{M_\omega}A_{\omega},\qquad\text{with}\qquad M_\omega:=\left\{U\in \D\,:\,\widetilde{S}(U)=d(\omega)\right\}.
\end{equation}
Moreover, whenever $d(\omega)>0$,
\begin{equation}
\label{eq-thirdequiv}
\left\{\begin{array}{l}\displaystyle A_{\omega}(U)=d(\omega)\\[.2cm]\displaystyle U\in M_\omega\end{array}\right.\qquad\Longleftrightarrow\qquad \left\{\begin{array}{l}\displaystyle S_{\omega}(U)=d(\omega)\\[.2cm]U\in N_\omega.\end{array}\right.
\end{equation}
\end{lemma}

\begin{proof}
Let us start by proving \eqref{eq-levA}. It is not restrictive to take $d(\omega)>0$ as the case $d(\omega)=0$ is trivial for \eqref{eq-levA}.

Preliminarily, we observe that $I_\omega(W)\geq 0$ for any $W\in M_\omega$. Indeed, assume by contradiction that $I_{\omega}(W)<0$. By Lemma
\ref{lem:dmu-alt}, $W$ cannot be a minimizer of $\widetilde{S}$ on
$\widetilde{N}_{\omega}$. Hence $\widetilde{S}(W)>d(\omega)$, which
contradicts the assumption $W\in M_\omega$.

Now, let $(U_n)_n$ be a minimizing sequence for $A_\omega$ on $M_\omega$, namely
$\widetilde{S}(U_n)=d(\omega)$ and
$A_\omega(U_n)\to \inf_{M_\omega}A_\omega$ as $n\to+\infty$.
Since $I_\omega(U_n)\geq 0$, using the second
equality in \eqref{equiv-form-Som}, we obtain
\begin{equation*}
	0\leq \left(\frac{1}{2}-\frac{1}{r}\right)I_\omega (U_n)
	=A_\omega(U_n)-\widetilde{S}(U_n)
	=\inf_{M_\omega} A_\omega+o(1)-d(\omega),
	\qquad\text{as}\quad n\to+\infty,
\end{equation*}
which yields $d(\omega)\leq \inf_{M_\omega} A_\omega$.

Conversely, let $(W_n)_n=\big((u_n,v_n)\big)_n$ be a minimizing sequence for $S_\omega$ on
$N_\omega$, namely $I_\omega(W_n)=0$ and
$S_\omega(W_n)\to d(\omega)$ as $n\to+\infty$. Using Lemma \ref{lem:dmu-alt} and $d(\omega)>0$, one can see that, for every $n$, there exists $\theta_n\in(0,1)$, with $\theta_n\to 1$ as $n\to+\infty$,
such that $\widetilde{S}(\theta_n W_n)=d(\omega)$. As a consequence,
$\theta_n W_n\in M_\omega$, which implies that
$A_\omega(\theta_n W_n)\geq \inf_{M_\omega} A_\omega$. Moreover, we have
\begin{equation*}
	0\leq I_\omega(\theta_n W_n)
	=\theta_n^2I_\omega(W_n)
	+(\theta_n^2-\theta_n^p)\|u_n\|_p^p
	+(\theta_n^2-\theta_n^r)\|v_n\|_r^r
	\leq C\left(\theta_n^2-\theta_n^{\max\{p,r\}}\right)
	\widetilde{S}(W_n)
\end{equation*}
and thus $I_\omega(\theta_n W_n)\to 0$ as $n\to+\infty$.
Combining this with the first equality in \eqref{equiv-form-Som}, we get
$S_\omega(\theta_n W_n)\to d(\omega)$ as $n\to+\infty$, and, using again
\eqref{equiv-form-Som}, we find that
\begin{equation*}
	0\leq\frac{1}{r}I_\omega(\theta_n W_n)
	=S_\omega(\theta_n W_n)-A_\omega(\theta_n W_n)
	\leq S_\omega(\theta_n W_n)-\inf_{M_\omega} A_\omega,
\end{equation*}
which gives $d(\omega)\geq \inf_{M_\omega} A_\omega$ and concludes the
proof of \eqref{eq-levA}.

We now prove \eqref{eq-thirdequiv}. If $S_\omega(U)=d(\omega)$ and
$U\in N_\omega$, then, by the equality between the first and the last
term in \eqref{equiv-form-Som}, it follows that $U\in M_\omega$ and
$A_\omega(U)=d(\omega)$. Conversely, if $U\in M_\omega$ and
$A_\omega(U)=d(\omega)$, then the second equality in
\eqref{equiv-form-Som} implies that $U\in N_\omega$. Combining this with
$U\in M_\omega$, we also obtain $S_\omega(U)=d(\omega)$, and the proof is
complete.
\end{proof}

\begin{remark}
Note that assuming $d(\omega)>0$ before \eqref{eq-thirdequiv} is not actually restrictive for our purposes because it is always satisfied in our application due to Lemma \ref{gs->min}.
\end{remark}

\begin{remark}
\label{rem-problemi}
 The idea of deriving two minimization problems equivalent to the minimization of the action functional on the Nehari's manifold is borrowed by \cite{ABCT-22}. However, in that case there is a quadratic strictly positive functional in place of $A_\omega$, which makes all the subsequent steps easier. Let us notice that the superquadratic nature of $A_\omega$ is ascribable to the presence of different nonlinearities on the half-line and on the plane.
\end{remark}

Now we can discuss the positivity of the restriction to the plane of the action minimizers.

\begin{lemma}
\label{phiom-pos}
Let $p,\,r>2$, $\beta>0$ and $\omega>E_{\rm{lin}}$. Let $U=(u,v)$ be an action minimizer at frequency $\omega$ and, given the decomposition $v=\phi_{\la}+q\G_{\la}$, assume that $q>0$ and $\phi_\la\not\equiv0$ for every $\la>0$. Then, $v(\x)>0$ for every $\x\in\Rd$. In particular, $\phi_{\la}(\x)>0$ for every $\x\in\Rd$ and every $\lambda\geq\omega$.
\end{lemma}

\begin{proof}
By \cite[Remark 2.1]{ABCT-22}, it is sufficient to prove the positivity of $\phi_\omega$. Then, let $U=(u,v)$ be an action minimizer at frequency $\omega$, and consider the decomposition $v=\phi_{\omega}+q\G_{\omega}$. By Lemma \ref{lem:thirdform}, $U$ is also a minimizer of $A_{\omega}$ on $M_\omega$.

Now, define $\Omega:=\{\x\in\Rd\setminus\{\z\}\,:\, \phi_{\omega}(\x)\neq 0\}$ and recall that $|\Omega|>0$. Hence, we can write $\phi_{\omega}(x)=e^{i\theta(\x)}|\phi_{\omega}(\x)|$, for every $\x\in \Rd\setminus\{0\}$, for some $\theta:\Omega\to[0,2\pi)$. Define, also, $\widehat{\Omega}:=\{\x\in\Omega\,:\, \theta(\x)\neq 0\}$ and assume, by contradiction, that $|\widehat{\Omega}|>0$. Consider the function $\widetilde{U}=(u,\widetilde{v})$, with $\widetilde{v}:=|\phi_{\omega}|+q\G_{\omega}$. On the one hand, since $\big|\na \phi_{\omega}(\x)\big|^{2}\geq\big|\na|\phi_{\omega}(\x)|\big|^{2}$ for a.e. $\x\in\R^2$, using the choice $\lambda=\omega$, we have that $Q_{\omega}(\widetilde{U})\leq Q_{\omega}(U)$, so that $A_{\omega}(\widetilde{U})\leq A_{\omega}(U)=d(\omega)$. On the other hand,
\begin{equation*}
|v(\x)|^{2}=|\phi_{\omega}(\x)|^{2}+q^{2}\G_{\omega}^2(\x)+2q\cos\big(\theta(\x)\big)|\phi_{\omega}(\x)|\G_{\omega}(\x)<\left|\widetilde{v}(\x)\right|^2\qquad\forall\,\x\in \widehat{\Omega},
\end{equation*} 
and thus, as $|\widehat{\Omega}|>0$, $\|\widetilde{v}\|_{L^{r}(\Rd)}>\|v\|_{L^{r}(\Rd)}$, whence
\begin{equation}
\label{p+r<dom}
\widetilde{S}(\widetilde{U})>\widetilde{S}(U)=d(\omega).
\end{equation}
Now, consider $\widetilde{U}_{\tau}=(u,\widetilde{v}_{\tau})$, with $\widetilde{v}_{\tau}:=\tau |\phi_\omega|+q\G_{\omega}$. It is straightforward that $A_{\omega}(\widetilde{U}_{\tau})<A_\omega(\widetilde{U})$, for every $\tau\in(-1,1)$. Note, also, that the function $\tau\mapsto \|\widetilde{v}_{\tau}\|_{L^{r}(\Rd)}^{r}$, given by
\begin{equation*}
\|\widetilde{v}_{\tau}\|_{L^{r}(\Rd)}^{r}=\int_{\Rd}\left|\tau^{2}|\phi_{\omega}(\x)|^{2}+q^{2}\G_{\omega}^2(\x)+2\tau q|\phi_{\omega}(\x)|\G_{\omega}(\x)\right|^{\f{r}{2}}\,d\x,
\end{equation*}
is continuous on $[-1,1]$. Assume, first, that $\theta(\x)\neq\pi$ for a.e. $\x\in\Omega$. Therefore,
\begin{equation*}
\|\widetilde{v}_{-1}\|_{L^{r}(\Rd)}^{r}<\|v\|_{L^{r}(\Rd)}^{r}<\|\widetilde{v}\|_{L^{r}(\Rd)}^{r},
\end{equation*}
and, hence, by the continuity of $\tau\mapsto\|\widetilde{v}_{\tau}\|_{L^{r}(\Rd)}^{r}$, there exists $\tau^\star\in (-1,1)$ such that
\begin{equation*}
\|\widetilde{v}_{\tau^\star}\|_{L^{r}(\Rd)}^{r}=\|v\|_{L^{r}(\Rd)}^{r}.
\end{equation*}
Assume, then, that $\theta(\x)=\pi$ for a.e. $\x\in\Omega$, i.e. that $v(x)=-|\phi_{\omega}(x)|+q\G_{\omega}(x)$ for a.e. $\x\in\Omega$. Since
\begin{equation*}
\left|\f{d|\widetilde{v}_{\tau}|^{r}}{d\tau}\right|=\left|r|\widetilde{v}_{\tau}|^{r-2}\widetilde{v}_{\tau}|\phi_{\omega}|\right|\leq r|\widetilde{v}|^{r-1}|\phi_{\omega}|\in L^{1}(\Rd)\qquad \forall\tau \in [-1,1],
\end{equation*}
by dominated convergence,
\begin{multline}
\label{d-dtau-vtau}
{\f{d}{d\tau}\|\widetilde{v}_{\tau}\|_{L^{r}(\Rd)}^{r}}_{\big|_{\tau=-1^{+}}}=r\int_{\Rd}\left|\widetilde{v}_{-1}(\x)\right|^{r-2}\widetilde{v}_{-1}(\x)|\phi_{\omega}(\x)|\,d\x\\
=r\int_{\Rd}\left|v(\x)\right|^{r-2}v(\x)|\phi_{\omega}(\x)|\,d\x.
\end{multline}
Moreover, by Remark \ref{rem-bs} $U$ satisfies \eqref{eq:plane}, namely it solves \eqref{EL-plane}. Now, setting $\chi=|\phi_{\omega}|$ and $\la=\omega$ in \eqref{EL-plane}, we get
\begin{equation}
\label{EL-appl}
\int_{\Rd}|v(\x)|^{r-2}v(\x)|\phi_{\omega}(\x)|\,d\x=-\int_{\R^{d}}\left|\na|\phi_{\omega}(\x)|\right|^{2}\,d\x-\omega\int_{\Rd}\left|\phi_{\omega}(\x)\right|^{2}\,d\x
\end{equation}
and thus, combined with \eqref{d-dtau-vtau}, there results
\begin{equation*}
{\f{d}{d\tau}\|\widetilde{v}_{\tau}\|_{L^{r}(\Rd)}^{r}}_{|\tau=-1^{+}}<0.
\end{equation*}
Since $\|\widetilde{v}_{1}\|_{L^{r}(\Rd)}^{r}>\|\widetilde{v}_{-1}\|_{L^{r}(\Rd)}^{r}=\|v\|_{L^{r}(\Rd)}^{r}$ and $\tau\mapsto\|\widetilde{v}_{\tau}\|_{L^{r}(\Rd)}^{r}$ is continuous, there exists $\tau^\star\in (-1,1)$ such that $\|\widetilde{v}_{\tau^\star}\|_{L^{r}(\Rd)}^{r}=\|v\|_{L^{r}(\Rd)}^{r}$.

Summing up, in both cases there exists $\widetilde{U}_{\tau^\star}=(u,\widetilde{v}_{\tau^\star})\in M_\omega$ such that $A_{\omega}(\widetilde{U})<A_{\omega}(U)$, but this contradicts the fact that $U$ is a minimizer of $A_{\omega}$ on $M_\omega$. Hence, $|\widetilde{\Omega}|>0$ is false and $\phi_{\omega}(\x)\geq 0$ for every $\x\in\Rd$.

It is, then, left to prove that in fact $\phi_{\omega}(\x)>0$ for every $\x\in\Rd$. Since $v$ solves \eqref{eq:plane}, $\phi_{\omega}$ satisfies
\begin{equation*}
(-\lap+\omega)\phi_{\omega}=|v|^{r-2}(\phi_{\omega}+q\G_{\omega}).
\end{equation*}
Therefore, as $\phi_\omega$ is nonnegative, $(-\lap+\omega)\phi_{\omega}\geq 0$ in $H^{-1}(\Rd)$ and so, as $\phi_{\omega}\not\equiv 0$ by Proposition \ref{prop:uv-neq-0}, by the Strong Maximum Principle (e.g., \cite[Theorem 3.1.2]{C-18}), there results in $\phi_{\omega}>0$ on $\Rd$, which concludes the proof.
\end{proof}

Finally, we can discuss the radial symmetry of the restriction to the plane of the action minimizers.

\begin{lemma}
\label{v-sym}
Let $p,\,r>2$, $\beta>0$ and $\omega>E_{\rm{lin}}$. Let $U=(u,v)$ be an action minimizer at frequency $\omega$ and, given the decomposition $v=\phi_{\la}+q\G_{\la}$, assume that $q>0$ and $\phi_\la>0$ for every $\la\geq\omega$. Then, $v$ is radially symmetric and decreasing along the radial direction. In particular, $\phi_{\la}$ is radially symmetric and decreasing along the radial direction for every $\lambda\geq\omega$.
\end{lemma}

\begin{proof}
By \cite[Remark 2.1]{ABCT-22} it is sufficient to prove that $\phi_{\omega}\equiv\phi_{\omega}^\star$, with $\phi_{\omega}^\star$ the symmetric decreasing rearrangement of $\phi_{\omega}$ (see e.g. \cite[Chapter 3]{LL-01}). Then, let $U=(u,v)$ be an action minimizer at frequency $\omega$, and consider the decomposition $v=\phi_{\omega}+q\G_{\omega}$. By Lemma \ref{lem:thirdform}, $U$ is also a minimizer of $A_{\omega}$ on $M_\omega$.

Assume by contradiction $\phi_{\omega}(\x)\neq \phi_{\omega}^\star(\x)$ for all $\x\in\Omega$, with $|\Omega|>0$. Define the function $\widetilde{U}:=(u,\widetilde{v})$, with $\widetilde{v}:=\phi_{\omega}^\star+q\G_{\omega}$. By \cite[Eqs. (31), (33) and (39)]{ABCT-22}, with the proviso that \cite[Eq. (33)]{ABCT-22} is an equality if and only if $g\equiv g^\star$, there results that $Q_{\omega}(\widetilde{U})\leq Q_{\omega}(U)$ and $\|\widetilde{v}\|_{L^{r}(\Rd)}^{r}>\|v\|_{L^{r}(\Rd)}^{r}$, so that
\begin{equation*}
A_{\omega}(\widetilde{U})\leq A_{\omega}(U)\qquad\text{and}\qquad\widetilde{S}(\widetilde{U})>\widetilde{S}(U).
\end{equation*}
Moreover, one can check that there exists $\tau\in (0,1)$ such that $\|\tau\phi_{\omega}^\star+q\G_{\omega}\|_{L^{r}(\Rd)}^{r}=\|\phi_{\omega}+q\G_{\omega}\|_{L^{r}(\Rd)}^{r}$. Besides, letting $U_\tau=(u,\tau\phi_{\omega}^\star+q\G_{\omega})$, there results that
\begin{equation*}
Q_{\omega}(U_\tau)<Q_{\omega}(\widetilde{U})
\end{equation*}
and thus
\begin{equation*}
A_{\omega}(U_\tau)<A_{\omega}(U)\qquad\text{and}\qquad\widetilde{S}(U_\tau)=\widetilde{S}(U),
\end{equation*}
which contradicts the fact that $U$ is a minimizer of $A_\omega$ on $M_\omega$ and concludes the proof.
\end{proof}

\begin{remark}
\label{rem-problemi2}
 The proof of Lemma \ref{v-sym} and of Lemma \ref{phiom-pos} require major modifications in the proof of the analogous results given in \cite{ABCT-22}. The reason is pointed out in Remark \ref{rem-problemi}. Indeed, the presence of a negative, non-quadratic term in $A_\omega$, forces variations that involve the sole restriction to the plane $v$. However, such an operation does not allow any control on the sign of $Q_\omega$ and, thus, any proper renormalization. Hence, we decided to act only on the regular part of $v$. This gives rise to  further technical issues (mainly in Proposition \ref{phiom-pos}), which can be  managed as showed before.
\end{remark}

Then, we can conclude the section with the proof of Theorem \ref{thm:4}, that follows combining previous results in this section.

\begin{proof}[Proof of Theorem \ref{thm:4}]
Let $U=(u,v)$ be a ground state for $E$ at mass $\mu$. In particular, fix $\la=1$, so that $v=\phi+q K_0/2\pi$, where $\phi$ stays for $\phi_1$. Then, by Proposition \ref{prop:uv-neq-0} $u\not\equiv 0$, $\phi\not\equiv 0$ and $q\neq 0$. Moreover, up to a multiplication by a constant phase, $q>0$ and $u(x)>0$ for every $x\geq 0$.
By Lemma \ref{lem:EL-eq}, $u$ solves \eqref{eq:half-line}, hence it is a proper translation of a one-dimensional soliton. On the other hand, \eqref{bc} follows by \eqref{first-cond} and \eqref{sec-cond} choosing $\la=1$. The properties of $v$ follows by combining Proposition \ref{prop:uv-neq-0} and Lemmas \ref{phiom-pos}, \ref{v-sym} and \ref{gs->min}.

 It is, then, left to prove that
\begin{equation}
\label{last-ineq}
E(U)<-\theta_p\mu^{\frac{p+2}{6-p}}.    
\end{equation}
Instead of proving \eqref{last-ineq}, we prove the stronger inequality
\[
E(U)<\min\{\Eps_\alpha(\mu,\Rp), \Eps_\rho(\mu,\Rd)\},
\]
since the right-hand side is not larger than $-\theta_p\mu^{\frac{p+2}{6-p}}$. Assume, by contradiction, that $E(U)=\Eps(\mu)=\min\{\Eps_\alpha(\mu,\Rp), \Eps_\rho(\mu,\Rd)\}$. Then, arguing as in the proof of Theorem \ref{thm:2}, we get  
\begin{equation*}
    E(U)=\inf_{U\in\D_\mu}\left\{E_\alpha(u,\Rp)+E_\rho(v,\Rd)\right\},
\end{equation*}
so that $U$ is also a ground state for $E_{\alpha}(\cdot,\Rp)+E_\rho(\cdot,\Rd)$ at mass $\mu$. However, by Theorem \ref{thm:3}, this entails that either $u\equiv0$ or $v\equiv0$, which is a contradiction and concludes the proof.
\end{proof}


\appendix
\section{The Laplacian on the hybrid plane}
\label{sec:operatorES}

Here we show that the Hamiltonian operators $H$ introduced by Exner and $\check{\text{S}}$eba in \cite{ES-87} to realize the Laplacian on the hybrid is associated to  the family of quadratic forms $Q$ defined in \eqref{quadratic}. In the notation $H$ and $Q$ we omit the dependence on the parameter $\alpha, \rho, \beta$, but it is clear that the symbol $H$ represents rather a family of operators and $Q$ a family of quadratic forms. The definition of $H$ presented below is slightly different from that of \cite{ES-87}, but it is easily seen that they are equivalent.

The operator $H$ is obtained as a self-adjoint operator that acts as the Laplacian on $\Rp$ and $\Rd$ outside the origin. Thus, $H$ is a self-adjoint extension of
\begin{equation}
\label{eq-h0}
 H_0:=H_{\Rp}\oplus H_{\Rd}
\end{equation}
with
\[
\begin{array}{ll}
\displaystyle H_{\Rp}: C^{\infty}_{c}(\Rp)
\to L^{2}(\Rp), & \quad H_{\Rp}u=-u''\\[.4cm]
\displaystyle H_{\Rd}: C^{\infty}_{c}(\Rd\setminus\{\z\})
\to L^{2}(\Rd), & \quad H_{\Rd}v=-\lap v.
\end{array}
\]
In \cite{ES-87}, the family of all extensions of $H_0$ is split up in five classes. Here we have considered class I only, which is a three-parameter family of extensions, since the others can be seen as degenerate cases (see Remark \ref{rem-degenerate}). 

Let $\alpha,\,\rho\in\R, \, \beta \in\C$. We define $H:D(H)\subset L^2(\I)\to L^2(\I)$ as the operator with domain
\begin{multline}
\label{dom-HSi}
D(H):=
\bigg\{U=(u,v)\in L^{2}(\I)\,:\,u\in H^{2}(\Rp), 
\\[.2cm]
v = \phi + q \frac{K_0}{2 \pi}, \, \phi \in H^{2}(\Rd), \, q \in \C \, \:\text{ and }\: \begin{pmatrix}
u'(0) \\
\phi(\z) 
\end{pmatrix}
= \begin{pmatrix}
\alpha & \beta \\
\overline{\beta} & \rho +\frac{\gamma-\log \, 2}{2\pi}
\end{pmatrix} 
\begin{pmatrix}
u(0) \\
q
\end{pmatrix}\bigg\},
\end{multline}
and action
\begin{equation}
\label{act-HSi}
H U=\left(-u'', - \Delta \phi - q \frac{K_0}{2 \pi}\right),\quad U=(u,v)\in D(H),
\end{equation}
We recall that the plane component of every function $U=(u,v)\in D(H)$ can be represented for every $\la>0$ as $v=\phi_\la+q\G_\la$ and for the action of $H$ one finds
$$ H U \ = \ (-u'',
-\lap \phi_\la-\la q \G_\la).$$ As a consequence, \eqref{dom-HSi} can be represented equivalently in terms of the dummy parameter $\lambda$. In this case boundary conditions in \eqref{dom-HSi} read
\begin{equation}
\label{BC}
\begin{cases}
u'(0)=\alpha u(0)+\beta q\\[.2cm]
\phi_\la(\z)={\overline{\beta}} u(0)+\left(\rho+\frac{\gamma-\log(2)+\log(\sqrt{\la})}{2\pi}\right)q.
\end{cases}
\end{equation}

\begin{remark}
Note that the boundary conditions in formula \eqref{BC} differ from those in formula \eqref{bc} of Theorem \ref{thm:4}. More precisely, while  formula \eqref{BC} involves both $\beta$ and its conjugate $\overline{\beta}$, in formula \eqref{bc} both terms are replaced by $-\beta$. This apparent inconsistency is due to the fact that, whereas the self-adjoint extensions above are derived for an arbitrary complex coupling parameter $\beta$, in the main body of the paper we have chosen to take $\beta\geq 0$ and to introduce a minus sign in front of it in order to improve readability. However, this is not restrictive in the study of the normalized the ground states of the energy. In order to see this, let us write the expression of the NLS energy on $\I$ consistent with \eqref{BC} for an arbitrary complex $\beta_0$:
\[
\widetilde E (U)=E_\alpha (u, \R^+) + E_\rho (v, \R^2) + \Re \left(\beta_0 q \overline{u(0)}\right).
\]
Since
\[
\beta_0\,q \overline{u(0)}=|\beta_0|\, q\,\overline{e^{-i\mathrm{Arg}\left(\beta_0\right)}\,u(0)}=-|\beta_0|\, q\,\overline{e^{-i\pi}\,e^{-i\mathrm{Arg}\left(\beta_0\right)}\,u(0)},
\]
for every $U=(u,v)\in\mathcal D_\mu$, if we define $U_{\beta_0}=(u_{\beta_0},v)$ with $u_{\beta_0}:=e^{-i\left(\pi+\mathrm{Arg}\left(\beta_0\right)\right)}u$, then $U_{\beta_0}\in\mathcal D_\mu$ and $\widetilde E(U)=E(U_{\beta_0})$, where $E$ is the energy defined by \eqref{e} with $\beta=|\beta_0|$.
\end{remark}

\begin{remark}
 \label{rem-degenerate}
 The other classes of self-adjoint extensions of $H_0$ established by \cite{ES-87} can be considered as degenerate since they are obtained by letting the parameters $\alpha,\rho,\beta$ go to infinity. More precisely, classes III, IV, and V correspond to the cases in which the half-line and the plane are decoupled. More precisely:
 \begin{itemize}[label=$\ast$]
  \item class III is obtained by letting $\alpha\to\infty$ and keeping $\rho$ and $\beta$ bounded;
  \item class IV is obtained letting $\rho\to\infty$ and keeping $\alpha$ and $\beta$ bounded;
  \item class V is obtained either letting $\alpha,\,\rho\to\infty$ and keeping $\beta$ bounded, or letting $|\beta|\to+\infty$ and keeping $\alpha$ and $\rho$ bounded.
 \end{itemize}

Class II, on the contrary, corresponds to a completely coupled case, in which one cannot decouple the two objects for any value of the parameters. It can be formally obtained, for instance, letting $|\alpha|,|\rho|,|\beta|\to+\infty$ in such a way that $\alpha\sim\rho\sim\beta$ and $|\beta|^2-\alpha\rho\sim|\beta|$ (e.g. $\beta>0$ and $\alpha=\rho=\beta+\sqrt{\beta}$).
\end{remark}

Let us show  that $Q:\D\to\R$, defined by \eqref{quadratic}, is the quadratic form associated with $H$. Preliminarily, one can directly check that $Q$ is symmetric and 
\begin{equation*}
Q(U)=\scal{U}{H U}_{L^{2}(\I)}\qquad \forall U\in D(H).
\end{equation*}
Hence,  to conclude it is sufficient to prove that $Q:\D\to L^{2}(\I)$ is associated with a unique self-adjoint operator. By \cite[Theorem VIII.15]{RS-80}, this is guaranteed if $Q(U)\geq -M\|U\|_{L^{2}(\I)}^{2}$, for some $M\in\R$, and $Q$ is complete in $\D$ with respect to the norm
 \begin{equation*}
\|U\|_{\D,M}:=\left(Q(U)+(1+M)\|U\|_{L^{2}(\I)}^{2}\right)^{\f{1}{2}}.
\end{equation*}
(see e.g. \cite[Chapter VIII.6]{RS-80}). To this aim, recalling \eqref{quadratic1}, \eqref{quadratic2}, \eqref {eq-bottom-2d} and the fact that 
\begin{equation}
\label{eq-bottom-1d}
-\ell_\alpha:=\inf_{\substack{u\in H^{1}(\Rp)\\ u\neq0}}\f{Q_{\alpha}(u,\Rp)}{\|u\|_{L^{2}(\Rp)}^{2}}=
\begin{cases}
\,0,\quad &\text{if}\quad\alpha \geq 0,\\
-\alpha^{2},\quad &\text{if}\quad\alpha<0.
\end{cases}
\end{equation}
The infimum is attained if and only if $\alpha<0$ by $Ce^{\alpha x}$, $C\in\C$ (see \cite[Section 6.2.2]{GTV-12}), there results
\begin{align*}
 Q(U) & \geq Q_{\alpha-|\beta|}(u,\Rp)+Q_{\rho-|\beta|}(v,\Rd)\\[.2cm]
             & \geq -\ell_{\alpha-|\beta|}\|u\|_{L^2(\Rp)}^2-\omega_{\rho-|\beta|}\|v\|_{L^2(\Rd)}^2\\[.2cm]
             & \geq -\max\big\{\ell_{\alpha-|\beta|},\omega_{\rho-|\beta|}\big\}\|U\|_{L^2(\I)}^2,\qquad\forall U=(u,v)\in\D.
\end{align*}
On the other hand, using again \cite[Chapter VIII.]{RS-80}, $Q$ is complete in $\D$ if and only if, for every sequence $(U_{n})_n\subset \D$ such that $U_{n}\to U$ in $L^{2}(\I)$ and $Q(U_{n}-U_{m})\to 0$, there results that $U\in \D$ and $Q(U_{n}-U)\to 0$. Then, fix $(U_{n})_n\subset \D$ with the two required properties. In particular, for any fixed $\la>0$, $U_{n}=(u_{n},v_{n})$, with $u_{n}\in H^{1}(\Rp)$ and $v_{n}-q_{n}\G_{\la}=:\phi_{n,\la}\in H^{1}(\Rd)$, for some $(q_{n})_n\subset \C$. Moreover, $u_{n}\to u$ in $L^{2}(\Rp)$, $v_{n}\to v$ in $L^{2}(\Rd)$ and $\|u_{n}-u_{m}\|_{L^{2}(\Rp)}$, $\|v_{n}-v_{m}\|_{L^{2}(\Rd)}$ and  $Q(U_{n}-U_{m})$ are uniformly bounded. By \eqref{gn1-inf} and choosing $\la=\omega_\rho\mathrm{e}^{4\pi}$, there results that
\begin{multline*}
\|u_{n}'-u_{m}'\|_{L^{2}(\Rp)}^{2}-C\|u_{n}'-u_{m}'\|_{L^{2}(\Rp)}\\
-C\la+|q_{n}-q_{m}|^{2}-C|q_{n}-q_{m}|\|u_{n}'-u_{m}'\|_{L^{2}(\Rp)}^{\f{1}{2}}\leq C,
\end{multline*}
from which, using a weighted Young's inequality, we deduce that both $\|u_{n}'-u_{m}'\|_{L^{2}(\Rp)}$ and $|q_{n}-q_{m}|$ are bounded. As a consequence,
\begin{equation*}
|u_n(0)-u_m(0)|^{2}\leq 2\|u_{n}'-u_{m}'\|_{L^{2}(\Rp)}\|u_{n}-u_{m}\|_{L^{2}(\Rp)}\to 0, 
\end{equation*}
and
\begin{equation*}
\left|\Re\left(\beta(q_{n}-q_{m})(\overline{u_{n}(0)}-\overline{u_{m}(0)})\right)\right|\leq C|q_{n}-q_{m}|\|u_{n}'-u_{m}'\|_{L^{2}(\Rp)}^{\f{1}{2}}\|u_{n}-u_{m}\|_{L^{2}(\Rp)}^{\f{1}{2}}\to 0.
\end{equation*}
Thus 
\begin{equation*}
\begin{split}
Q(U_{n}-U_{m})=&\|u_{n}'-u_{m}'\|_{L^{2}(\Rp)}^{2}+\|\na\phi_{n,\la}-\na\phi_{m,\la}\|_{L^2(\Rd)}^{2}\\
&+\la\|\phi_{n,\la}-\phi_{m,\la}\|_{L^2(\Rd)}^{2}+|q_{n}-q_{m}|^{2}+o(1)\to 0,
\end{split}
\end{equation*}
which entails that $(u_{n})_{n}$ is a Cauchy sequence in $H^{1}(\Rp)$, $(\phi_{n,\la})_{n}$ is a Cauchy sequence in $H^{1}(\Rd)$ and $(q_{n})_{n}$ is a Cauchy sequence in $\C$. Hence, there results that $U\in \D$ and that $Q(U_{n}-U)\to 0$.


\section{Eigenvalues of the Laplacian on the hybrid plane}
\label{sec:Elin}
Here we give the discrete spectrum of the operator $H$ introduced in Appendix \ref{sec:operatorES}, in view of the importance of its least eigenvalue, that coincides with the quantity $-E_{\rm{lin}}$, defined in \eqref{Elin}. 

The eigenvalues of $H$ are the numbers $\nu\in\R$ for which there exists $U=(u,v)\in D(H)\setminus\{(0,0)\}$ such that
\begin{equation}
\label{eq-eigeneq}
 (H-\nu)U=(0,0), \qquad\text{i.e.}\qquad
 \left\{
 \begin{array}{ll}
  -u''-\nu u=0, & \text{on }\Rp,\\[.2cm]
  -\lap\phi_\la-\nu\phi_\la-q(\nu+\la)\G_\la=0, & \text{on }\Rd.
 \end{array}
 \right.
\end{equation}
The case $\beta=0$ is immediate as the problem decouples, and the discrete spectrum consists of $\{-\omega_\rho\}$, when $\alpha\geq0$, and of $\{-\omega_\rho,-\alpha^2\}$, when $\alpha<0$, where the eigenspace of $-\omega_\rho$ is spanned by $(0,\G_{\omega_\rho})$ and, when $\alpha<0$, the eigenspace of $-\alpha^2$ is spanned by $\mathrm{e}^{\alpha x}$. On the other hand, if $\beta>0$, then one finds that the eigenvalues are the negative solutions in $\nu$ of
\begin{equation}
\label{eq-eigenrel}
 \left(\alpha+\sqrt{-\nu}\right)=\left(\rho+\frac{\gamma-\log(2)+\log(\sqrt{-\nu})}{2\pi}\right)^{-1}\beta^2.
\end{equation}
Such solutions are not explicit, but an easy qualitative study of the functions on the two sides of \eqref{eq-eigenrel}, yields that
\begin{itemize}[label=$\ast$]
 \item for every $\alpha,\,\rho\in\R$, $\beta>0$, there exists a solution, denoted by $\ell^1$, smaller than $\min\{-\ell_\alpha,-\omega_\rho\}$,
 \item for every $\rho\in\R$, $\alpha<0$, $\beta>0$, there exists another solution, denoted by $\ell^2$, belonging to $\left(\max\{-\ell_\alpha,-\omega_\rho\},0\right)$.
\end{itemize}
Hence, the discrete spectrum consists of $\{\ell^1\}$, when $\alpha\geq0$, and of $\{\ell^1,\ell^2\}$, when $\alpha<0$. The eigenspaces of $\ell^j$ are spanned by $(\mathrm{e}^{-\sqrt{-\ell^j} x},\f{\alpha+\sqrt{-\ell^j}}{\beta}\G_{-\ell^j}(\x))$. 

Summing up, the discrete spectrum of $H$ reads
\[
 \sigma_{\mathrm{d}}(H)=\left\{
 \begin{array}{ll}
  \{-\omega_\rho\}, & \text{if }(\alpha,\beta)\in[0,+\infty)\times\{0\}\\[.2cm]
  \{-\omega_\rho,-\alpha^2\}, & \text{if }(\alpha,\beta)\in\R^-\times\{0\}\\[.2cm]
  \{\ell^1\}, & \text{if }(\alpha,\beta)\in[0,+\infty)\times\Rp\\[.2cm]
  \{\ell^1,\ell^2\}, & \text{if }(\alpha,\beta)\in\R^-\times\Rp.
 \end{array}
 \right.
\]
and $-E_{\rm{lin}}$, defined by \eqref{Elin}, is actually the least eigenvalue of $H$, i.e. $-E_{\rm{lin}}=\min\sigma_{\mathrm{d}}(H)$.


\section*{Statements and Declarations}


\subsection*{Fundings and acknowledgments} F.B. and L.T. have been partially supported by the INdAM GNAMPA project 2023 ``Modelli nonlineari in presenza di interazioni puntuali'' (CUP E53C22001930001).

R.A., R.C. and L.T. acknowledge that this study was carried out within the project E53D23005450006 ``Nonlinear dispersive equations in presence of singularities'' -- funded by European Union -- Next Generation EU within the PRIN 2022 program (D.D. 104 - 02/02/2022 Ministero dell'Universit\`a e della Ricerca). This manuscript reflects only the author's views and opinions and the Ministry cannot be considered responsible for them.


\end{document}